\documentclass[openacc]{rsproca_new}

\newtheorem{theorem}{Theorem}[section]
\newtheorem{lemma}[theorem]{Lemma}
\newtheorem{definition}[theorem]{Definition}
\newtheorem{example}[theorem]{Example}
\newtheorem{remark}[theorem]{Remark}

\newtheorem{notation}[theorem]{Notation}

\newtheorem{manualtheoreminner}{Theorem}
\newenvironment{manualtheorem}[1]{%
  \renewcommand\themanualtheoreminner{#1}%
  \manualtheoreminner
}{\endmanualtheoreminner}

\usepackage{anyfontsize}
\usepackage{lmodern}
\usepackage[T1]{fontenc}
\usepackage{amsfonts}
\usepackage{amssymb}
\usepackage{latexsym}
\usepackage{mathrsfs}
\usepackage{amsmath}
\usepackage{amsthm}
\usepackage{graphicx}
\usepackage{afterpage}
\usepackage{tikz}
\usepackage{pgfplots}
\usepackage{pgfplotstable}
\pgfplotsset{compat = newest}
\newcommand*{\diam}{\mathrm{diam}}
\newcommand*{\dist}{\mathrm{d}}
\definecolor{myaqua}{RGB}{15,205,255}
\titlehead{Research}

\begin{document}

\title[Fractal dimensions for Iterated Graph Systems]{Fractal dimensions for Iterated Graph Systems}

\author{Nero Ziyu Li}

\address{Department of Mathematics, Huxley Building, Imperial College London, London, SW7 2AZ, United Kingdom}

\subject{fractal geometry, graph theory, complex networks}


\corres{Nero Ziyu Li\\
\email{z5222549@zmail.unsw.edu.au, ziyu.li21@imperial.ac.uk}}

\begin{abstract}
Building upon~\cite{Li2023}, this study aims to introduce fractal geometry into graph theory, and to establish a potential theoretical foundation for complex networks.
Specifically, we employ the method of substitution to create and explore fractal-like graphs, termed deterministic or random iterated graph systems. While the concept of substitution is commonplace in fractal geometry and dynamical systems, its analysis in the context of graph theory remains a nascent field.

By delving into the properties of these systems, including diameter and distal, we derive two primary outcomes. Firstly, within the deterministic iterated graph systems, we establish that the Minkowski dimension and Hausdorff dimension align analytically through explicit formulae. Secondly, in the case of random iterated graph systems, we demonstrate that almost every graph limit exhibits identical Minkowski and Hausdorff dimensions numerically by their Lyapunov exponents.

The exploration of iterated graph systems holds the potential to unveil novel directions.
These findings not only, mathematically, contribute to our understanding of the interplay between fractals and graphs, but also, physically, suggest promising avenues for applications for complex networks. 
%\absbreak
\end{abstract}

%\rsbreak

\maketitle

%Insert `2000 Mathematics Subject Classification' numbers here:
%\classno{05C12, 05C82(primary), 28A80(primary), 60G07}

%\extraline{The author has been supported by the Additional Funding Programme for Mathematical Sciences, delivered by EPSRC (EP/V521917/1) and the Heilbronn Institute for Mathematical Research, and also by the EPSRC Centre for Doctoral Training in Mathematics of Random Systems: Analysis, Modelling and Simulation (EP/S023925/1).}

%\maketitle

\tableofcontents
%\titlecontents{subsection}[3cm]{\bf \large}{\contentslabel{2.5em}}{}{\titlerule*[0.5pc]{$\cdot$}\contentspage\hspace*{3cm}}

%\begin{highlights}
%\item We establish random iterated graph system.
%\item We obtain corresponding Minkowski.
%\item We prove Minkowski and Hausdorff dimensions for such systems are consistent.
%\end{highlights}

%\begin{keyword}
%Substitution systems \sep Random graphs \sep Complex networks \sep Fractality \sep Matrix combinatorics
%\end{keyword}

%\tableofcontents
%\markboth{On the fractal dimensions for deterministic and random iterated graph systems}{}

\section{Introduction and main results}
\subsection{Introduction}
Graph theory originated in the 18th century and saw rapid development throughout the 20th century, meanwhile fractal geometry was proposed and explored during the 20th century as well. 
Both fields have undergone rigorous mathematical development to this day.
The integration of these two fields is generally associated with the study of complex networks:
this is primarily due to the work of Barabási on scale-free networks~\cite{Barabasi99} and Song's exploration~\cite{SoHaMa05} of fractal networks, which demonstrate how these mathematical concepts can be applied to understand and describe the structure of various complex systems.
%The studies on the graphs with self-similarity are often related to complex networks since~\cite{SoHaMa05}.
%Such a property called fractality for complex networks is used to describe the existence of Minkowski dimension for graphs. 

However, few attempts have been made to develop the theory: most relevant studies are based on real-world data and finite networks, rather than mathematical objects.
Such a gap is to be filled, and it is of great interest to explore the behaviour of infinite graph fractals.

% In respect of fractal theory, the fractal-like graphs are not well discussed.
% Since Hutchinson\cite{Hutch81} set up the iterated function systems (IFS) strictly in a general framework, large numbers of notable works contribute to the studies of self-similar sets, for instance directed-graph IFS~\cite{Maul88,Edg92}, IFS with overlaps~\cite{Fan00,Lau01,Sid07} and random IFS~\cite{Fal86,Hutch00}.
% However, the analysis of fractals regarding graph theory remains an immature field.

To establish the fractals for graphs, we utilise the idea of substitution and Iterated Function Systems (IFS), where these two concepts are more than common in dynamical systems and fractal geometry.
It is also always natural to relate substitutions to dynamics because a natural source of substitutions is tiling and substitutive dynamical system~\cite{Rau82,Thur89,Fogg02,Barge06}.
Since Hutchinson\cite{Hutch81} set up the iterated function systems (IFS) strictly in a general framework, large numbers of notable works contribute to the studies of self-similar sets, for instance directed-graph IFS~\cite{Maul88,Edg92}, IFS with overlaps~\cite{Fan00,Lau01,Sid07} and random IFS~\cite{Fal86,Hutch00}.

Building on these foundations, we rigorously define the Iterated Graph System by iteratively substituting the (coloured) arcs in a graph (randomly) with certain fixed graphs. 
However, the concept of the iterated graph system is not entirely new; it evolves from the predecessor known as substitution networks, which were first introduced by Xi et al.~\cite{XiWaWaYuWa17}, where fractality and scale-freeness are obtained.
Li et al.~\cite{LiYuXi18,LiYaWa19} brought coloured graphs into this model and the average distance was studied by Ye et al.~\cite{XiYe19b,YeXi19a} in which the existence of average distance is exhibited under some circumstances.
Iterated graph systems can also be viewed both as the graph version of graph-directed fractals~\cite{mauldin1988hausdorff,barnsley1989recurrent} and also as a generalisation of them because the substitution process inherent to iterated graph systems includes a graph-directed methodology. 
We will discuss this issue in another work.

%Also see recent works~\cite{Rao16,Mal18,Nagai21}.

Therefore, from a mathematical perspective, studying iterated graph systems is at least intriguing: we anticipate that fractals on graphs will exhibit unique behaviours not seen in Euclidean spaces.
This is indeed the case, as evidenced by the specially defined degree dimension~\cite{Li2023} for graph fractals. 
Physically, iterated graph systems provide a solid mathematical foundation for the study of complex networks, especially those with scale-freeness and fractality.
The potential applications include considerations of fractal dimensions, degree distributions, average distance, random walks, diffusion process, and percolation, among others.
For instance, a quick corollary in this paper is that for any network characterised by a specific fractal dimension, it is possible to correspondingly design an iterated graph system that mirrors this dimension.
%The substitution networks are basically a class of growing graphs where every arc keeps substituted by some graphs following specific rules.
%See Fig.~\ref{fig:substitution} for example.
%This model constructs a sequence of expanding self-similar patterns, as is essentially a very analogue of IFS.
%On the other hand, the substitution networks can also be reckoned as generalised L-systems for graphs.
%Several previous results have exhibited some crucial properties for substitution networks.
% Substitution systems have appeared implicitly and incidentally throughout much literature in graph theory and in a wide range of research that incorporates mathematical modelling~\cite{Rozenfeld07,SoShHe06}.
% As the explicit objects of study, substitution networks were first introduced by Xi et al.~\cite{XiWaWaYuWa17}, where fractality and scale-freeness are obtained.
% Li et al.~\cite{LiYuXi18,LiYaWa19} brought coloured graphs into this model and proved the fractality by ergodic theory, but Li failed to analyze the Minkowski dimension.
% The average distance was studied by Ye et al.~\cite{XiYe19b,YeXi19a} where the limit of average distance would converge under some circumstances.
%Li~\cite{Li2023} defined the degree dimension for infinite graphs and proved that random substitution networks exactly satisfy scale-freeness and have the degree dimension.

\subsection{Main results}
The deterministic (or random) iterated graph systems roughly refer to the system generating infinite graphs (randomly) iteratively by independent coloured arc substitutions. 
Find Fig.~\ref{fig:example0} as an example.
Rigorous definitions will be given in Section~\ref{sec:deterministic iterated graph systems} and Subsection~\ref{sec:random iterated graph systems}.

%%%%%%%%%%%%
\newcommand{\redarc}[2]{\draw[draw=red,             thick,>=stealth,->] (#1) -- (#2);\fill[black] (#1) circle (1.5pt);\fill[black] (#2) circle (1.5pt);}
\newcommand{\bluearc}[2]{\draw[draw=myaqua!90!black,thick,>=stealth,->] (#1) -- (#2);\fill[black] (#1) circle (1.5pt);\fill[black] (#2) circle (1.5pt);}

\newcommand{\mynode}[2]{\begin{scope}[shift={(#1)},scale=#2]\node[node] (0,0) () {};\end{scope}}
\newcommand{\Rrule}[4]{\begin{scope}[shift={(#1)},scale=#2,rotate=#3]
  \foreach \nn/\x/\y in {0/0/0, 1/1/0, 2/1/1, 3/2/0, 4/2/1, 5/3/0}{\node[node] (\nn) at (\x,\y) {};}
  \draw[black] (0) -- (5) (1) -- (2) (3) -- (4);
  \foreach \nn in {0,...,5}{\node[node] () at (\nn) {};}#4\end{scope}}
\newcommand{\RRrule}[3]{\begin{scope}[shift={(#1)},scale=#2,rotate=#3]
    \Rrule{0,0}{1}{  0}{}
    \Rrule{6,0}{1}{180}{}
    \Rrule{6,0}{1}{  0}{}
    \Rrule{3,3}{1}{270}{}
    \Rrule{6,0}{1}{ 90}{}\end{scope}}
    
\newcommand{\Rrulea}[4]{\begin{scope}[shift={(#1)},scale=#2,rotate=#3]
  \foreach \na in {0,...,4}{\coordinate (\na) at (72*\na-162:1.577) {};}
  \draw[draw=red,>=stealth,->] (0) -- (1);
  \draw[draw=red,>=stealth,->] (1) -- (2);
  \draw[draw=red,>=stealth,->] (4) -- (3);
  \draw[draw=myaqua!90!black,>=stealth,->] (3) -- (2);
  \draw[draw=myaqua!90!black,>=stealth,->] (0) -- (4);
  \foreach \na in {1,3,4}{\fill[black] (\na) circle (1.5pt);}
  \foreach \na in {0,2}{\fill[black] (\na) circle (2.25pt);}#4\end{scope}}

\newcommand{\Rruleb}[4]{\begin{scope}[shift={(#1)},scale=#2,rotate=#3]
\foreach \na in {0,...,5}{\pgfmathsetmacro\x{2.12*cos(18*\na+45)}
\pgfmathsetmacro\y{2.12*sin(18*\na+45)-1.5}
\coordinate (\na) at (\x,\y) {};}
\foreach \na in {6,...,9}{\pgfmathsetmacro\x{2.12*cos(18*\na+135)}
\pgfmathsetmacro\y{2.12*sin(18*\na+135)+1.5}
\coordinate (\na) at (\x,\y) {};}
  \draw[draw=red,>=stealth,->] (9) -- (0);
  \draw[draw=red,>=stealth,->] (1) -- (2);
  \draw[draw=red,>=stealth,->] (3) -- (2);  
  \draw[draw=red,>=stealth,->] (3) -- (4);
  \draw[draw=red,>=stealth,->] (4) -- (5);
  \draw[draw=myaqua!90!black,>=stealth,->] (0) -- (1); \draw[draw=myaqua!90!black,>=stealth,->] (5) -- (6);
  \draw[draw=myaqua!90!black,>=stealth,->] (7) -- (6);
  \draw[draw=myaqua!90!black,>=stealth,->] (8) -- (7); \draw[draw=myaqua!90!black,>=stealth,->] (8) -- (9);
  \foreach \na in {0,...,9}{\fill[black] (\na) circle (1.5pt);}
  \foreach \na in {0,5}{\fill[black] (\na) circle (2.25pt);}#4\end{scope}}
  
  \newcommand{\RanRrule}[5]{\begin{scope}[shift={(#1)},scale=#2,rotate=#3]#4
    \fill[black] (0,0) circle (2.25pt);\draw (0,-1) node(){};
    \fill[black] (3,0) circle (2.25pt);#5\end{scope}}

\newcommand{\Raa}[4]{\RanRrule{#1}{#2}{#3}{\redarc{ 1,0}{0,0}\redarc{ 2,1}{2,0}\redarc{2,1}{3,0}
\bluearc{1,0}{1,1}\bluearc{1,0}{2,0}\bluearc{3,0}{2,0}}{#4}}
\newcommand{\Rab}[4]{\RanRrule{#1}{#2}{#3}{\bluearc{1,0}{0,0}\bluearc{1,0}{2,0}\bluearc{3,0}{2,0}}{#4}}
\newcommand{\Rba}[4]{\RanRrule{#1}{#2}{#3}{\redarc{ 0,0 }{1.5,-.75}\redarc{1.5,0}{0,0}\redarc{3,0}{1.5,0}
\redarc{1.5,.75}{3,0}\bluearc{1.5,.75}{0,0}\bluearc{3,0}{1.5,-.75}}{#4}}
\newcommand{\Rbb}[4]{\RanRrule{#1}{#2}{#3}{\redarc{1.5,0}{ 0,0}\redarc{60:1.5}{1.5,0}
\bluearc{0,0}{60:1.5}\bluearc{1.5,0}{3,0}}{#4}}

\newcommand{\ABnodes}{\draw (0,-.25) node[below] {\makebox(0,0){\footnotesize $A$}};\draw (3,-.25) node[below] {\makebox(0,0){\footnotesize $B$}};}

\renewcommand{\Rrulea}[4]{\begin{scope}[shift={(#1)},scale=#2,rotate=#3]
  \foreach \na in {0,...,4}{\coordinate (\na) at (72*\na-162:1.577) {};}
  \draw[draw=red,>=stealth,->] (0) -- (1);
  \draw[draw=red,>=stealth,->] (2) -- (1);
  \draw[draw=red,>=stealth,->] (3) -- (4);
  \draw[draw=myaqua!90!black,>=stealth,->] (2) -- (3);
  \draw[draw=myaqua!90!black,>=stealth,->](4) -- (0);
  \foreach \na in {1,3,4}{\fill[black] (\na) circle (1.5pt);}
  \foreach \na in {0,2}{\fill[black] (\na) circle (2.25pt);}#4\end{scope}}

\renewcommand{\Rruleb}[4]{\begin{scope}[shift={(#1)},scale=#2,rotate=#3]
  \foreach \na in {0,...,5}{\pgfmathsetmacro\x{2.12*cos(18*\na+45)}
                            \pgfmathsetmacro\y{2.12*sin(18*\na+45)-1.5}
                            \coordinate (\na) at (\x,\y) {};}
  \foreach \na in {6,...,9}{\pgfmathsetmacro\x{2.12*cos(18*\na+135)}
                            \pgfmathsetmacro\y{2.12*sin(18*\na+135)+1.5}
                            \coordinate (\na) at (\x,\y) {};}
  \draw[draw=red,>=stealth,->] (9) -- (0); \draw[draw=red,>=stealth,->] (1) -- (2);
  \draw[draw=red,>=stealth,->] (2)-- (3);
  \draw[draw=red,>=stealth,->] (4)-- (3);
  \draw[draw=red,>=stealth,->]  (4)-- (5);
  \draw[draw=myaqua!90!black,>=stealth,->] (0) -- (1);
  \draw[draw=myaqua!90!black,>=stealth,->](5) -- (6);
  \draw[draw=myaqua!90!black,>=stealth,->](6)-- (7);
  \draw[draw=myaqua!90!black,>=stealth,->](8)-- (7);
  \draw[draw=myaqua!90!black,>=stealth,->](9) -- (8);
  \foreach \na in {0,...,9}{\fill[black] (\na) circle (1.5pt);}
  \foreach \na in {0,5}{\fill[black] (\na) circle (2.25pt);}#4\end{scope}}

%%%%%%%%%%%%

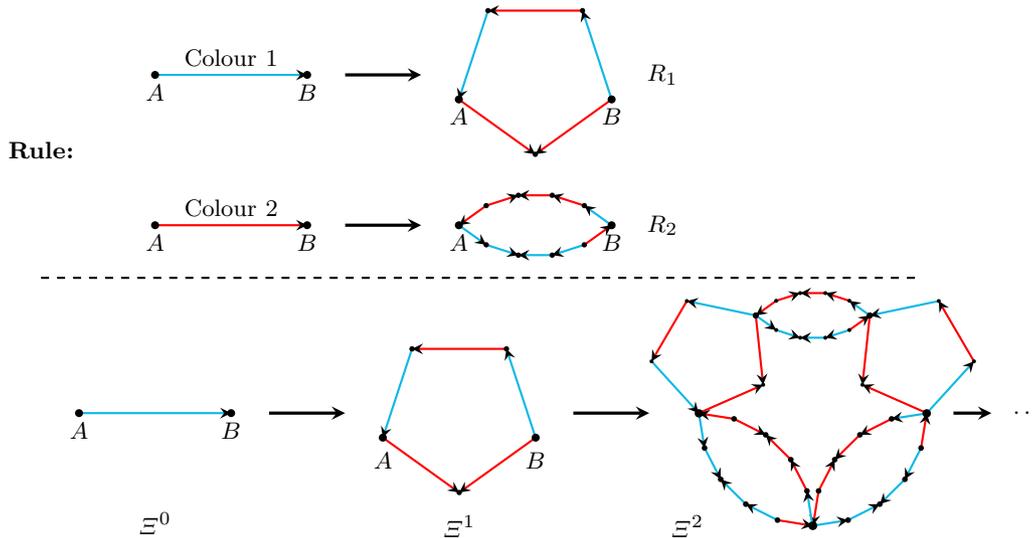
\begin{figure}[ht]
\centering
\begin{tikzpicture}[scale=1,black,thick]
  \begin{scope}[scale=1,shift={(1,2)}]
  \draw (-1.5,-1) node {\normalsize \textbf{Rule:}};
    \draw[draw=myaqua!90!black,>=stealth,->] (0,0) -- (2,0)  node[midway,above]{Colour 1};
    \fill (0,0) circle (1.5pt) (2,0) circle (1.5pt);
    \draw (0,0) node[below] {\footnotesize $A$};
    \draw (2,0) node[below] {\footnotesize $B$};
    \draw[very thick,>=stealth,->] (2.5,0) --++ (1,0);
    \Rrulea{5,0}{.67}{0}{\draw (-162:1.577) node[below] {\footnotesize $A$};
                         \draw (-18: 1.577) node[below] {\footnotesize $B$};
                         \draw (2.5,0) node {\makebox(0,0){$R_{1}$}};}
  \end{scope}
  \begin{scope}[scale=1,shift={(1,0)}]
    \draw[draw=red,>=stealth,->] (0,0) -- (2,0)  node[midway,above]{Colour 2};
    \fill (0,0) circle (1.5pt) (2,0) circle (1.5pt);
    \draw (0,0) node[below] {\footnotesize $A$};
    \draw (2,0) node[below] {\footnotesize $B$};
    \draw[very thick,>=stealth,->] (2.5,0) --++ (1,0);
    \Rruleb{5,0}{.67}{0}{\draw (-1.5,0) node[below] {\footnotesize $A$};
                         \draw ( 1.5,0) node[below] {\footnotesize $B$};
                         \draw ( 2.5,0) node {\makebox(0,0){$R_{2}$}};}
  \draw[dashed] (-1.5,-.7) -- (10,-.7);
  \end{scope}
  \begin{scope}[scale=1,shift={(0,-2.5)}]
    \draw[draw=myaqua!90!black,>=stealth,->] (0,0) -- (2,0) node[midway,below=1.2cm]{$\Xi^0$};
    \fill (0,0) circle (1.5pt) (2,0) circle (1.5pt);
    \draw (0,0) node[below] {\footnotesize $A$};
    \draw (2,0) node[below] {\footnotesize $B$};
    \draw (-.75,0) node {};
    \draw[very thick,>=stealth,->] (2.5,0) --++ (1,0);
    \Rrulea{5,0}{.67}{0}{};
    %\draw (342:.8) node;
    %\draw[gray!75!white,->] (198:1.35) .. controls (270:1.35) .. (342:1.35);}
    \draw (4,-.4) node[below] {\footnotesize $A$};
    \draw (6,-.4) node[below] {\footnotesize $B$};
    \draw (5,-1.5) node {$\Xi^1$};
    \draw[very thick,>=stealth,->] (6.5,0) --++ (1,0);
    \begin{scope}[shift={(9.65,0)},scale=.5]
      \Rrulea{150:3.085}{1}{ 60}{}
      \Rrulea{ 30:3.085}{1}{300}{}
      \Rruleb{  0,2.6  }{1}{  0}{}
      \Rruleb{ 1.5,-1.5}{1.41}{45}{}
      \Rruleb{-1.5,-1.5}{1.41}{315}{}
      \fill (-3,0) circle (2.25pt) (3,0) circle (2.25pt);
      %\draw[gray!75!white,->] (180:2.5) .. controls (150:1) .. (120:2.5) ..
      %  controls (90:1.65) .. ( 60:2.5) .. controls ( 30:1) .. (0:2.5);
      \draw ( 90:1.2) node(){{}};
    \end{scope}
    %\draw (8    ,0) node[below] {\footnotesize $A$};
    %\draw (11.25,0) node[below] {\footnotesize $B$};
    \draw (8,-1.5) node {$\Xi^2$};
    \draw[very thick,>=stealth,->] (11.5,0) --++ (0.5,0);
    \draw (12.5,0) node {$\dots$};
    
  \end{scope}
\end{tikzpicture}
\caption{An example of deterministic iterated graph system}
\label{fig:example0}
\end{figure}

This paper mainly contains two conclusions, the first of which solves the problem left in~\cite{LiYaWa19}: 

\begin{manualtheorem}{\ref{thm:Deterministic Substitution}}
Given a deterministic primitive iterated graph system and its graph limit $\Xi$, we have
\[
\dim_B(\Xi)
=\dim_H(\Xi)
= \dfrac{\log\rho(\mathbf{M})}{\log \rho_{\min}(\mathcal{D})}\,.
\]
\end{manualtheorem}
where $\rho(*)$ represents the spectral radius.
Relevant definitions will be given in Subsection~\ref{sec:deterministic iterated graph systems} and Section~\ref{Section: Dimension}, and the proof will be presented in Section~\ref{Section: Proof of fractality} later.

Another result is to construct random iterated graph systems and prove the associated graph limits $\Xi$ satisfy the following property.

\begin{manualtheorem}{\ref{thm:Random Substitution}}
%\label{thm:Random Substitution}
Given a random primitive iterated graph system and one of its graph limits $\Xi$, we have
\[
\mathbb{P} \Big(
\dim_B (\Xi) =  \dim_H (\Xi)
= \frac{ \mathcal{L}(\mathcal{M})}{ \min_{\mathcal{D}\in\mathscr{D}} \mathcal{L}(\mathcal{D})}
\Big)=1 \,,
\]
\end{manualtheorem}
where $\mathcal{L}(*)$ is the Lyapunov exponent.
Related definitions will be addressed in Subsection~\ref{sec:random iterated graph systems}, and the proof will be shown in Subsection~\ref{sec:Proof of Random}.

\bigskip

The concept of iterated graph systems is straight and simple.
However, the difficulty and novelty of this paper lie in the solution of a combinatorial matrix problem (\textbf{Subsection~\ref{subSection: A combinatorial matrix problem}}), estimation of diameter and distal (\textbf{Subsection~\ref{subSection: Estimation of distance}}) and analysis of random substitution (\textbf{Section~\ref{sec:Random substitution}}).

%Moreover, this study is substantially also an investigation of arc-substituted L-system for graphs.
%These results will be found useful for generating random fractal networks, as there are only few rigorous methods to do so.

%%%%%%%%%%%%%%%%%%%%%%%%%%%%%%%%%%%%%%%%
\section{Deterministic iterated graph systems}\label{sec:Preliminaries for deterministic iterated graph systems}

\subsection{Minkowski dimension for graphs}
\label{Section: Dimension}
The Minkowski dimension is also known as box-counting dimension.
Graphs with consistent weight of arcs naturally construct metric spaces. 
However, the Minkowski dimension on graphs does have particular properties.
In this paper, we only consider weakly connected graphs containing no self-loops or multiple edges but with directions and consistent weights.
{\em Infinite graphs} are graphs with infinitely many arcs.
Infinite graphs can be an induced limit of a convergent graph sequence $G^0, G^1, \ldots, G^n, \ldots, G^\infty$.

We consider weights on edges and arcs. 
For any weighted undirected graph~$G$, 
the {\em weighted distance} between any two nodes $a$ and $b$ in~$G$ is
\begin{align*}
  \dist_G(a,b) = 
  \inf\Big\{ \sum \mu(e) \::\: e \in E(P) \text{ where $P$ is a path from $a$ to $b$ in $G$}\Big\}\,.
\end{align*}
For any finite weighted undirected graph $G$,
define the {\em weighted diameter} $\Delta(G)$ to be the maximal distance between any two nodes in~$G$:
\[
  \Delta(G) = \sup\{\dist_G(a,b): a,b \in V(G)\}\,.
\]
In this paper, the distances $\dist_G(a,b)$ and diameter $\Delta(G)$ are {\em always} extended to finite weighted directed graphs~$G$
simply by ignoring the arc directions and just considering the underlying undirected graph~$\underline{G}$, as follows:
\begin{align*}
  \dist_G(a,b)&:= \dist_{\underline{G}}(a,b)\\\text{and}\qquad
  \Delta(G) &:=   \Delta(\underline{G})\,.
\end{align*}
These definitions also extend to unweighted undirected (resp., directed) graphs by 
assigning the weight 1 to each of their edges (resp., arcs); 
then $\dist_G(a,b)$ and $\Delta(G)$ are the usual (unweighted) distances and diameter for graphs~$G$.

Let $G$ be any finite graph, either directed or undirected and either weighted or unweighted.
The {\em scaled graph $\hat{G} = (V(\hat{G}),E(\hat{G}))$} of $G = (V(G),E(G))$ is 
the weighted graph with the same nodes and arcs/edges as~$G$,
and with arc/edge weights $\mu(e) = \frac{1}{\Delta(G)}$ for all $e\in E(G)$ to form that $\Delta(\hat{G}) = 1$.  
The graph $\hat{G}^\infty$ can also be infinite, in which case each arc weight $\mu(e)$ equals~0.

Therefore,
\begin{equation}
\label{eq:ellL}
N_{\ell}(\hat{G})
=N_{L}(G)
\end{equation}
while $\ell=L/\Delta(G)$.

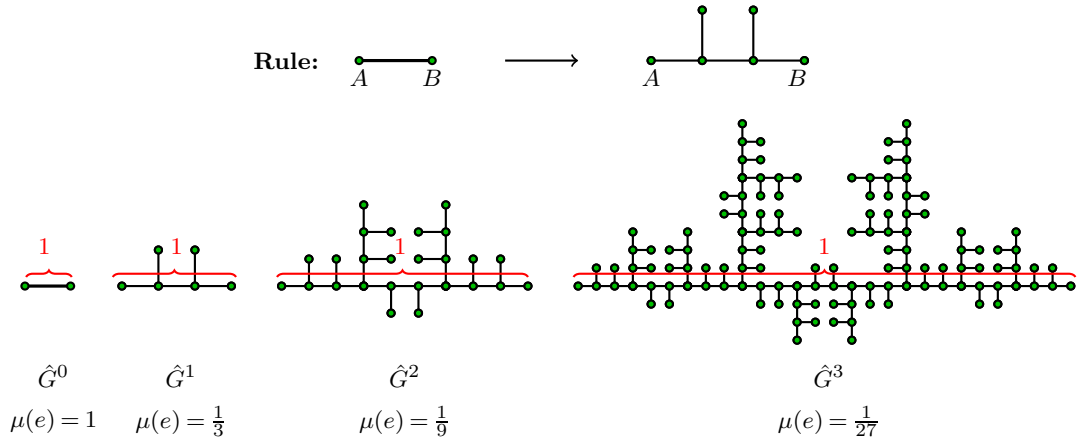
\begin{figure}[ht]
\centering\vspace*{40mm}\makebox(0,0){%
\begin{tikzpicture}[shift={(-12,0)},scale=.6,black,thick]
  \tikzstyle{node}=[circle,fill=green!75!black,draw=black ,inner sep = 0.33mm, outer sep = 0mm]
  
  \begin{scope}[shift={(0,1)},scale=1.6]
\draw (-4,0) node {\normalsize \textbf{Rule:}};
\tikzstyle{node}=[circle,fill=green!75!black,draw=black ,inner sep = 0.33mm, outer sep = 0mm]
\draw[very thick] (-3,0) -- (-2,0);
\mynode{-3,0}{20}
\mynode{-2,0}{20}
\draw (-3,0) node[below] {\footnotesize $A$};
\draw (-2,0) node[below] {\footnotesize $B$};
\draw (1,0) node[below] {\footnotesize $A$};
\draw (3,0) node[below] {\footnotesize $B$};
\draw[->,thick] (-1,0)--(0,0);
\Rrule{1,0}{0.7}{0}{}
  \end{scope}
  
  \begin{scope}[shift={(-12.125,-4)},scale=1]
    \draw[very thick] (0,0) -- (1,0);
    \mynode{0,0}{5}
    \mynode{1,0}{5}
  \end{scope}
  
  \Rrule{-10,-4}{.8}{0}{}
  \RRrule{-6.5,-4}{.6}{0}
  
  \begin{scope}[shift={(0,-4)},scale=.4]
    \RRrule{0,0}{1}{  0}{}
    \draw[decorate,decoration={brace,mirror},red] (-27.75,0.5) -- (-30.25,0.5);
    \draw (-29.25,2.2)[red] node {$1$};
    \draw (-28.75,-5) node {$\hat{G}^0$};
    \draw (-28.75,-7.5) node {$\mu(e)=1$};
    \RRrule{18,0}{1}{180}{}
    \draw[decorate,decoration={brace,mirror},red] (-18.75,0.5) -- (-25.5,0.5);
    \draw (-22,2.2)[red] node {$1$};
    \draw (-21.75,-5) node {$\hat{G}^1$};
    \draw (-21.75,-7.5) node {$\mu(e)=\frac{1}{3}$};
    \RRrule{18,0}{1}{  0}{}
    \draw[decorate,decoration={brace,mirror},red] (-2.75,0.5) -- (-16.5,0.5);
    \draw (-9.74,2.2)[red] node {$1$};
    \draw (-9.5,-5) node {$\hat{G}^2$};
    \draw (-9.5,-7.5) node {$\mu(e)=\frac{1}{9}$};
    \RRrule{9,9}{1}{270}{}
    \draw[decorate,decoration={brace,mirror},red] (27.25,0.5) -- (-0.25,0.5);
    \draw (13.5,2.2)[red] node {$1$};
    \draw (13.75,-5) node {$\hat{G}^3$};
    \draw (13.75,-7.5) node {$\mu(e)=\frac{1}{27}$};
    \RRrule{18,0}{1}{ 90}{}
  \end{scope}
%  \draw[white] (8,0) circle (1pt);  
\end{tikzpicture}}
\vspace*{30mm}
\caption{Scaled graphs associated to iterated graph systems}
\label{fig:unit substitution}
\end{figure}

\begin{example}\label{ex:substitution-network}{\rm 
An example of scaled graphs is given in Fig.~\ref{fig:unit substitution}, 
which shows the first four scaled graphs $\hat{G}^0, \hat{G}^1, \hat{G}^2, \hat{G}^3$
associated to the iterated graph system with $\hat{G}^0$ as initial graph 
and with $\hat{G}^1$ as the substitution rule graph.
For simplicity, arcs are here drawn as undirected edges.
As $\Delta(G^n) = 3^n$, the weight of each arc in $\hat{G}^n$ is scaled to~$\frac{1}{3^n}$.}
\end{example}

\begin{definition}
Let $N_{\ell}(\hat{G})$ be the minimum number of boxes 
with weighted diameter strictly {\em less} than $\ell$ needed to cover the vertices set of graph~$\hat{G}$. 
That is, $V(G) \subset \bigcup V(U_i) $ for the family of boxes $\{U_i\}$.
\end{definition}

\begin{definition}\label{def:Fractality for Graphs}{\rm
The Minkowski dimension is defined as
\begin{equation*}\label{eq:graphlimit}
  \dim_B(G) := \dim_B(\hat{G})=
  \lim_{\ell\to\,0}\frac { \log N_{\ell}(\hat{G})}{-\log \ell}
\end{equation*}
if such a limit exists.
We say that a graph $G$ has the {\em fractality} property 
if and only if the Minkowski dimension of $G$ exists and is positive.
In this case, we call $G$ a {\em graph fractal}.}
\end{definition}

%%%%%%%%%%%%%%%%%%%%%%%%%%%%%%%%%%%%%%%%%%%%%%%%

%This section is to introduce the random iterated graph systems model and corresponding main results.
\subsection{Deterministic iterated graph systems}\label{sec:deterministic iterated graph systems}

\begin{definition}
A double $\mathscr{R}=(\Xi^0,\mathcal{R})$ is called a {\em deterministic iterated graph system} if
\begin{gather*}
\Xi^0 \text{ is a finite directed graph} \\
\text{and }\mathcal{R}=\{R_{i}\}_{i=1}^{\lambda} \text{ is a family of directed graphs where } \lambda\in\mathbb{N}.
\end{gather*}
\end{definition}

Recall Fig.~\ref{fig:example0} for example.
In this paper, let the number of colours be $\lambda \in \mathbb{N}$.
%For each $i=1,\ldots,\lambda$, set positive integers $q_i$, which is the number of possible rule graphs of colour $i$. 
Each directed rule graph $R_{i}$ has a node $A$ and a node $B$ that respectively replace the beginning node $A$ and ending node $B$ of $e$;
this will determine exactly how $R_{i}$ replaces~$e$.
In addition, in this paper we always require $\dist_{R_{i}}(A,B)\geq 2$ for all $i$.

To construct a sequence of growing graphs, start from the initial graph $\Xi^0$, which is usually just a single arc (for instance a blue arc in Fig.~\ref{fig:example0}).
We construct $\Xi^1$ by replacing all $k$-coloured arc in $\Xi^0$ by $R_{k}$.
Note here, as stated above, the substitutions regarding node $A$ and $B$ are unique.
We then iteratively replace all $k$-coloured arcs for all $k$ in $\Xi^1$ to obtain $\Xi^2$.
In this way, denote the graph after $n\in \mathbb{N}$ iterations by $\Xi^n$;
accordingly, we have a sequence of graphs $\{\Xi^n\}_{n=0}^n$.
We also call the graph sequence $\Xi^0,\Xi^1,\ldots,\Xi^n,
\ldots,\lim_{n\to \infty}\Xi^n=\Xi$ {\em substitution networks}.
We call the graph limit $\Xi$ a {\em graph limit} for the substitution networks or for a deterministic iterated graph system (existence proved by Subsection~\ref{subsection: Existence}).
%For simplicity, write $\mathcal{G}^\infty$ for the collection of all possible $\Xi$.

%The example in Figure~\ref{fig:example0} is also used in~\cite{Li2023} where a different property called scale-freeness is studied.
%In that paper, scale-freeness is considered as a special self-similarity in terms of degree distribution.
%These results exhibit the enormous interests and potential properties of random iterated graph systems.
%and particularly let $\Xi$ denotes such graph. 

\subsection{Relevant definitions}
%Before strictly presenting the main results, several definitions are to be introduced in advance.

\begin{notation}{\rm
For any $k$-dimensional vector $\mathbf{x}=(x_1,\dots,x_n)\in\mathbb{R}_+^{k}$, 
write $[\mathbf{x}]_i=x_i$ for each $i = 1,\ldots,k$.
Define $\mathbf{\boldsymbol{\chi}}(G)$ to be the vector whose $j$-th entry is
the number of $j$-coloured arcs in graph $G$.
For instance, $\chi(\Xi^1)=(2,3)$ in Fig.~\ref{fig:example0}.
let $\mathbf{u}\leq \mathbf{v}$ denote that $[\mathbf{u}]_i\leq [\mathbf{v}]_i$ for all $i=1,\ldots,k$.
}
%Also, let $\boldsymbol{\chi}(v)$ be the $2n$-dimensional row vector whose $(2j-1)$-th and $2j$-th entries is the $j$-coloured out-degree and in-degree of node $v$, respectively.}
\end{notation}

\begin{definition}\label{def:arc matrix}{\rm
Define the $\lambda \times \lambda$ matrix $\mathbf{M} = ({m}_{ij})$ with entries
\[
  {m}_{ij} = [\boldsymbol{\chi}(R_{i})]_j \,.
\]}
%Each substitution step replaces each $i$-coloured arc by a random subgraph that, on average over~$k$, contains $\overline{m}_{ij}$ $j$-coloured arcs.}
\end{definition}

\begin{definition}\label{def:Path Matrix D}{\rm
Recall that a path in a graph is {\em simple} if it has no repeating nodes.

Donate a path between nodes $A$ and $B$ in underlying (undirected) $R_i$ by $P_i$, and
\[
\mathcal{P}_i=\big\{P_i \::\: P_i \subset R_i \big\}.
\]

%Let $\mathcal{P}_{i}$ be the set of subgraphs of $R_{i}$ whose underlying (undirected) graph is a path, donated by $P_i$, between~$A$ and~$B$.
Consider the Cartesian product 
\[
  \mathcal{P}= \prod_{j=1}^{\lambda} \mathcal{P}_{j}\,.
\]
Each element (or say choice) $\mathcal{C}= (P_{1},\ldots,P_{\lambda})\in\mathcal{P}$ is a vector of length $\lambda$ each whose entries $P_{i}$ is in $\mathcal{P}_{i}$.
%For each such vector $\mathcal{C}$, define
%\[
%  \boldsymbol{\nu}(\mathcal{C}) = %\boldsymbol{\chi}(P_{i})\,.
%\]
Define, for each $i=1,\ldots,\lambda$,
\[
  \mathcal{V}_i = \big\{ \boldsymbol{\chi}(P_i) \::\: P_i \in \mathcal{P}_i\big\}\,.
\]
Let $\mathcal{D}$ be the set of all matrices 
\[
    \mathbf{D}_{\mathcal{C}}
  = \begin{pmatrix}\mathbf{d}_1\\[-1mm]\vdots\\\mathbf{d}_\lambda\end{pmatrix}
  = \begin{pmatrix}\boldsymbol{\chi}(\mathcal{C}_1)\\[-1mm]\vdots\\\boldsymbol{\chi}(\mathcal{C}_\lambda)\end{pmatrix}\,
\]
where $\mathbf{d}_i\in \mathcal{V}_i$ for each $i=1,\ldots,\lambda$ or, 
equivalently, $\mathcal{C}_i\in\mathcal{P}_i$. 
Let $\rho(*)$ be spectral radius, and $\rho_{\min}(\mathcal{D})$ %, or simply $\rho_{\min}$,
denote the smallest spectral radius of matrices in $\mathcal{D}$:
\[
  \rho_{\min}(\mathcal{\mathcal{D}}) := \min\big\{\rho(\mathbf{D}) \::\: \mathbf{D}\in\mathcal{D}\big\}
\]
Also define
\[
  \mathcal{D}_{\min}=\{\mathbf{D} \::\: \rho(\mathbf{D})=\rho_{\min}(\mathcal{D}),\mathbf{D}\in\mathcal{D}\}.
\]

If $\mathbf{M}$ and all matrices in $\mathcal{D}$ are primitive (or positive), we say the iterated graph systems is {\bf primitive}.
}
\end{definition}

\begin{theorem}[main theorem]\label{thm:Deterministic Substitution}
Given a deterministic primitive iterated graph system $\mathcal{R}$ and graph limit $\Xi$, we have
\[
\dim_B(\Xi)
=\dim_H(\Xi)
= \dfrac{\log\rho(\mathbf{M})}{\log \rho_{\min}(\mathcal{D})}\,.
\]
\end{theorem}

\begin{remark}
    For any network with a given algebraic number fractal dimension, we can always directly construct a deterministic iterated graph system that possesses the same dimension. 
\end{remark}

\begin{figure}[ht]
\centering
\begin{tikzpicture}[scale=1,black,thick]
  \begin{scope}[scale=1,shift={(1,2)}]
  \draw (-1.5,-1) node {\normalsize \textbf{Rule:}};
    \draw[draw=myaqua!90!black,>=stealth,->] (0,0) -- (2,0) node[midway,above]{Colour 1};
    \fill (0,0) circle (1.5pt) (2,0) circle (1.5pt);
    \draw (0,0) node[below] {\footnotesize $A$};
    \draw (2,0) node[below] {\footnotesize $B$};
    \draw[very thick,>=stealth,->] (2.5,0) --++ (1,0);
    \Rrulea{5,0}{.67}{0}{\draw (-162:1.577) node[below] {\footnotesize $A$};
                         \draw (-18: 1.577) node[below] {\footnotesize $B$};
                         \draw (2.5,0) node {\makebox(0,0){$R_{1}$}};}
  \end{scope}
  \begin{scope}[scale=1,shift={(1,0)}]
    \draw[draw=red,>=stealth,->] (0,0) -- (2,0)node[midway,above]{Colour 2};
    \fill (0,0) circle (1.5pt) (2,0) circle (1.5pt);
    \draw (0,0) node[below] {\footnotesize $A$};
    \draw (2,0) node[below] {\footnotesize $B$};
    \draw[very thick,>=stealth,->] (2.5,0) --++ (1,0);
    \Rruleb{5,0}{.67}{0}{\draw (-1.5,0) node[below] {\footnotesize $A$};
                         \draw ( 1.5,0) node[below] {\footnotesize $B$};
                         \draw ( 2.5,0) node {\makebox(0,0){$R_{2}$}};}
  \draw[dashed] (-1.5,-.7) -- (10,-.7);
  \end{scope}
  \begin{scope}[scale=1,shift={(0,-2.5)}]
    \draw[draw=myaqua!90!black,>=stealth,->] (0,0) -- (2,0) node[midway,below=1.2cm]{$\Xi^0$};
    \fill (0,0) circle (1.5pt) (2,0) circle (1.5pt);
    \draw (0,0) node[below] {\footnotesize $A$};
    \draw (2,0) node[below] {\footnotesize $B$};
    \draw (-.75,0) node {};
    \draw[very thick,>=stealth,->] (2.5,0) --++ (1,0);
    \Rrulea{5,0}{.67}{0}{};
    %\draw (342:.8) node;
    %\draw[gray!75!white,->] (198:1.35) .. controls (270:1.35) .. (342:1.35);}
    \draw (4,-.4) node[below] {\footnotesize $A$};
    \draw (6,-.4) node[below] {\footnotesize $B$};
    \draw (5,-1.5) node {$\Xi^1$};
    \draw[very thick,>=stealth,->] (6.5,0) --++ (1,0);
    \begin{scope}[shift={(9.65,0)},scale=.5]
      \Rrulea{150:3.085}{1}{ 60}{}
      \Rrulea{ 30:3.085}{1}{300}{}
      \Rruleb{  0,2.6  }{1}{  0}{}
      \Rruleb{ 1.5,-1.5}{1.41}{45}{}
      \Rruleb{-1.5,-1.5}{1.41}{315}{}
      \fill (-3,0) circle (2.25pt) (3,0) circle (2.25pt);
      %\draw[gray!75!white,->] (180:2.5) .. controls (150:1) .. (120:2.5) ..
      %  controls (90:1.65) .. ( 60:2.5) .. controls ( 30:1) .. (0:2.5);
      \draw ( 90:1.2) node(){{}};
    \end{scope}
    %\draw (8    ,0) node[below] {\footnotesize $A$};
    %\draw (11.25,0) node[below] {\footnotesize $B$};
    \draw (8,-1.5) node {$\Xi^2$};
    \draw[very thick,>=stealth,->] (11.5,0) --++ (0.5,0);
    \draw (12.5,0) node {$\dots$};
  \end{scope}
\end{tikzpicture}
\caption{An example of deterministic iterated graph systems}
\label{fig:example1}
\end{figure}
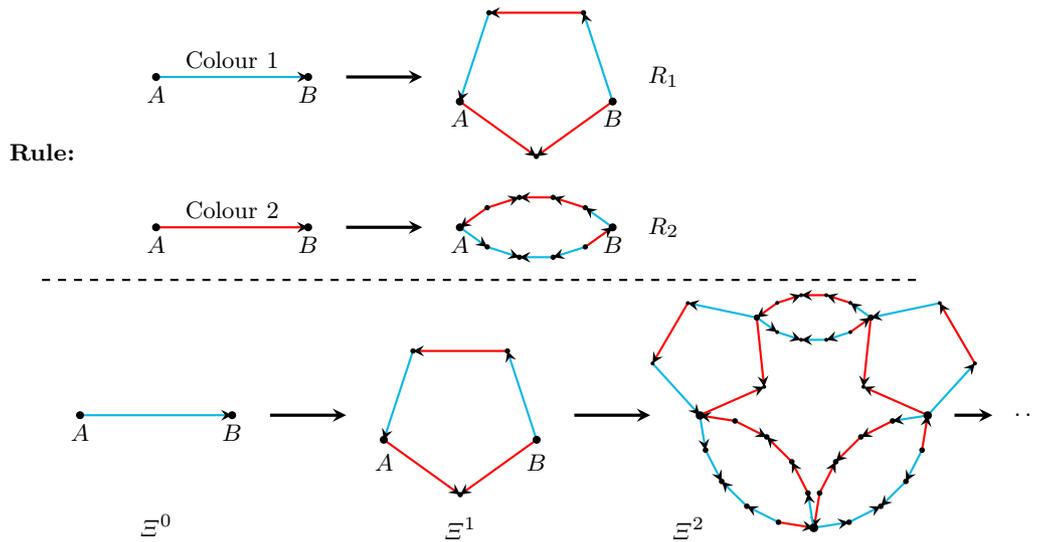

\begin{example}\label{exa:D}{\rm
Recall the iterated graph system of Fig.~\ref{fig:example1}.
For Definition~\ref{def:arc matrix}, we have
\[
\mathbf{M} = 
\begin{pmatrix}   
2 & 3 \\
5 & 5
\end{pmatrix}
\]
with $\rho(\mathbf{M}) \approx 7.\,6533$.

For Definition~\ref{def:Path Matrix D}, there are $|\mathcal{P}_{1}| = 2$ paths in $R_{1}$, say $P_{1}$ and $P_{1}'$,
where path $P_{1}$ contains 2 blue edges and 1 red edge;
that is, $\boldsymbol{\chi}(P_{1}) = (2,1)$,
and where path $P_{1}'$ contains 0 blue and 2 red edges; that is, $\boldsymbol{\chi}(P_{1}') = (0,2)$.
In $R_{2}$, there is also two paths $P_{2}$ and $P_{2}'$. $P_{2}$ contains 1 blue edge and 4 red edges;
that is, $\boldsymbol{\chi}(P_{2}) = (1,4)$.
Similarly, $\boldsymbol{\chi}(P_{2}') = (4,1)$.
Hence,
\[
\mathcal{V}_1 
      = \biggl\{ \Bigl(2,1\Bigr)\,,\; \Bigl(0,2\Bigr)\biggr\} \text{ and }
\mathcal{V}_2
      = \biggl\{ \Bigl(1,4\Bigr)\,,\; \Bigl(4,1\Bigr)\biggr\}\,. 
\]
As a result, $\mathcal{D}$ contains $|\mathcal{V}_1||\mathcal{V}_2| = 2 \times 2 = 4$ matrices, which are
\[
\mathcal{D}=
\biggl\{
\begin{pmatrix}   
2 & 1 \\
1 & 4
\end{pmatrix}
,
\begin{pmatrix}   
2 & 1 \\
4 & 1
\end{pmatrix}
,
\begin{pmatrix}   
0 & 2 \\
1 & 4
\end{pmatrix}
,
\begin{pmatrix}   
0 & 2 \\
4 & 1
\end{pmatrix}
\biggr\}\,.
\]
The smallest spectral radius of these matrices is
$\rho_{\min}(\mathcal{D}) = \frac{1}{2}\sqrt{33}+\frac{1}{2} \approx 3.3723$
and a matrix in $\mathcal{D}$ with this spectral radius is
\[
  \mathcal{D}_{\min} = \Biggl\{
  \begin{pmatrix}
  0 & 2 \\
  4 & 1\end{pmatrix}\Biggr\}
        \,.
\]}
As a conclusion, by Theorem~\ref{sec:deterministic iterated graph systems} we obtain that for the demonstrated example
\[
\dim_B(\Xi)=\dim_H(\Xi)=
\frac{\log 7.6533}{\log 3.3723}
= 1. 6742\,.
\]

Modelling the network by the volume-greedy ball-covering algorithm (VGBC) as given by Wang et al.~\cite{WaWaXiChWaBaYuZh17},
we obtain the simulated values given in Fig.~\ref{fig:one simulation}.
It shows a simulation of Minkowski dimension of $\Xi^5$, providing an estimated Minkowski dimension~$1.6219$.

%%%%%%%%%%%%%%%%%%%%%%%%%%
\begin{figure}
\centering
\pgfplotstableread{
X1	Y1
10	0.025797643
11	0.015665421
12	0.015665421
13	0.020048865
14	0.014156367
15	0.014156367
16	0.012791032
17	0.008407588
18	0.008407588
19	0.008407588
20	0.007042254
21	0.007042254
22	0.005533199
23	0.005533199
24	0.005533199
25	0.004024145
26	0.004024145
27	0.004024145
28	0.004024145
29	0.004024145
30	0.003592986
31	0.003592986
32	0.003592986
33	0.003592986
34	0.003592986
35	0.002443231
36	0.002443231
37	0.002443231
38	0.002299511
39	0.002299511
40	0.002299511
41	0.002299511
42	0.002083932
43	0.002083932
44	0.002083932
45	0.002083932
46	0.002083932
47	0.002083932
48	0.001868353
49	0.002083932
50	0.002083932
}\mytable
\begin{tikzpicture}[scale=0.8]
    \begin{axis}[
        xmode=log,
        ymode=log,
        xmin = 8, xmax = 100,
        ymin = 0.001, ymax = 0.1,
        width = 0.9\textwidth,
        height = 0.675\textwidth,
       % xtick distance = 1,
       % ytick distance = 1,
        grid = both,
        minor tick num = 1,
        major grid style = {lightgray},
        minor grid style = {lightgray!25},
        xlabel = {$L$},
        ylabel = {$\frac{N_L(\Xi^n)}{|V(\Xi^n)|}$},
        legend cell align = {left},
        legend pos = north east
        ]
        %\foreach \x in;
        \addlegendentry{Estimated Minkowski dimension is $1.6219$};
        \addplot[color=purple,mark=square,only marks] table[x = X1, y = Y1] {\mytable};
        \addplot[thick, black] table[ x = X1, y = {create col/linear regression={y=Y1}} ] {\mytable};
    \end{axis}
\end{tikzpicture}
\caption{Test of fractality when $n=5$}
\label{fig:one simulation}
\end{figure}
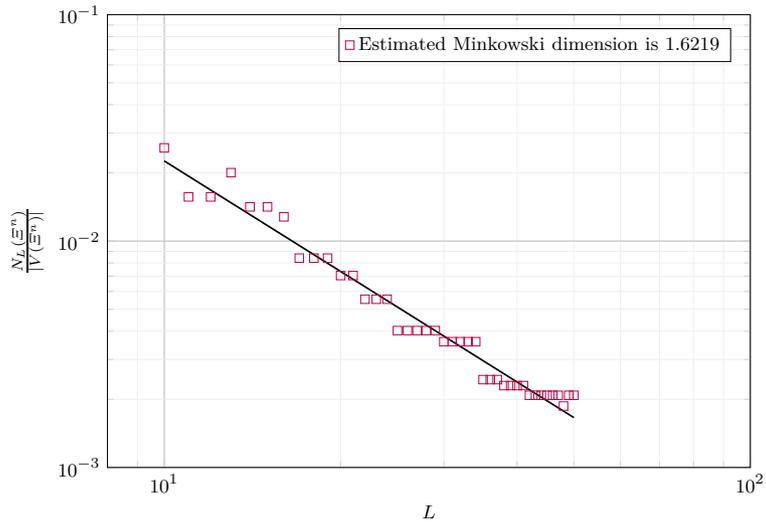

%%%%%%%%%%%%%%%%%%%%%%%%%%%%%
\end{example}

\smallskip
We will later prove Theorem~\ref{thm:Deterministic Substitution} through the following two theorems.
\begin{theorem}\label{thm:part I}
We will prove in Subsection~\ref{Subsection: Proof of fractality} that
\[
\dim_B(\Xi)=\frac{\log \rho(\mathbf{M})}{\log \rho_{\min}(\mathcal{D})} \,.
\]
\end{theorem}

\begin{theorem}\label{thm:Part II}
We will prove in Subsection~\ref{Subsection:Proof of Hausdorff} that
\[
\dim_H(\Xi)=\frac{\log \rho(\mathbf{M})}{\log \rho_{\min}(\mathcal{D})} \,.
\]
\end{theorem}

\subsection{Growth rates of arcs and nodes}
This section introduces several necessary tools and properties for the proof of fractality in Section~\ref{Section: Proof of fractality}.
\textit{We will omit the proofs of such results that may be found in~\cite{LiYaWa19,Li2023}.}

% \begin{notation}
%     For simplicity, let $PF(k,l)=\{\mathbf{X}\in\mathbb{R}_+^{k\times l}\::\: \mathbf{X} \geq 0 \}$.
% \end{notation}

\begin{lemma}\label{lem:Norm Range}
Let $\mathbf{X}$ be a primitive non-negative $k \times k$ matrix with spectral radius $\rho(\mathbf{X})$.
Then, for any positive vector $\mathbf{u}\in\mathbb{R}_+^k$ and for any $n\in \mathbb{N}$,
\[
  c^{-1} \rho(\mathbf{X})^n \leq \Vert  \mathbf{u}\mathbf{X}^n\Vert_1 \leq c\rho(\mathbf{X})^n\,,
\]
where $c>1$ is a constant depending on $\mathbf{u}$ and $\mathbf{X}$.
\end{lemma}

\begin{notation}
$\sim$ is defined as asymptotic equivalence.
We write $f(x)  \overset{x \to x_0}{\asymp} g(x)$ if $\lim_{x \to x_0}{f(x)}/{g(x)}=c$ where $c>0$.
We may omit $x \to x_0$ if no ambiguity occurs.
\end{notation}

\begin{lemma}\label{lem:Norm Limit}
Let $\mathbf{X}$ be a primitive non-negative $k \times k$ matrix with spectral radius $\rho(\mathbf{X})$.
Then, for any positive vector $\mathbf{u}\in\mathbb{R}_+^n$,
\[
  \Vert  \mathbf{u}\mathbf{X}^n\Vert_1\overset{t\to \infty}{\asymp} \rho(\mathbf{X})^n\,.
\]
\end{lemma}

%\subsection{Growth of arcs and nodes}
%\label{subSection: cardinalities}
Let $\mathbf{M}^{(n)} = ({m}_{ij}^{(n)})$ be the $\lambda \times \lambda$ matrix
whose entries ${m}_{ij}^{(n)}$ equal the  number of $j$-coloured arcs
that result from substituting arcs $t$ times, first by substituting an $i$-coloured arc
and then $n-1$ times substituting the subsequently resulting arcs of all colours.

\medskip

\begin{lemma}\label{lem:Arc Transition Matrix}
For each $n\geq 1$, $\mathbf{M}^{(n)} = \mathbf{M}^n$.
\end{lemma}
\begin{proof}
By induction.
\end{proof}

%For any real vectors $\mathbf{u} = (u_1,\ldots,u_\lambda)$ and $\mathbf{v} = (v_1,\ldots,v_\lambda)$,
%\begin{notation}
%let $\mathbf{u}\leq \mathbf{v}$ denote that $[\mathbf{u}]_k\leq [\mathbf{v}]_k$ for all $k=1,\ldots,n$.
%\end{notation}

\begin{lemma}\label{lem:Arc Growth}
$|E(\Xi^n)|\overset{n \to \infty}{\asymp} \rho(\mathbf{M})^n$.
\end{lemma}
\begin{proof}
By $\chi(\Xi^n)=\chi(\Xi^{n-1})\mathbf{M}$.
\end{proof}

%\subsection{Growth of nodes}
\begin{lemma}\label{lem:Node Growth}
$|V(\Xi^n)|\overset{n \to \infty}{\asymp} \rho(\mathbf{M})^n$.
\end{lemma}
\begin{proof}
Define
\[
 V^*(\Xi^n) = \{v \::\: v\in V(\Xi^n), v\notin V(\Xi^{n-1})\}\,.
\]
Then$|V(\Xi^n)|=\sum_{i=0}^{n} |V^*(\Xi^i)|$.
Notice all new nodes $V^*(\Xi^n)$ in $V(\Xi^n)$ are in fact generated by substituting arcs in $\Xi^{n-1}$.
Hence, for all $n\in \mathbb{N}$,
\begin{equation*} 
  |V^*(\Xi^n)| = \boldsymbol{\chi}(\Xi^{n-1})
\begin{pmatrix}
|V(R_{1})| -2 \\
\vdots\\
|V(R_{\lambda})| -2 
\end{pmatrix}
\,.
\end{equation*}
Given the initial graph $|V(\Xi^0)|=2$,
\begin{equation*} 
|V(\Xi^n)|=2+\sum_{i=1}^{n} (\boldsymbol{\chi}(\Xi^{i-1}))
\begin{pmatrix}
|V(R_{1})| -2 \\
\vdots\\
|V(R_{\lambda})| -2 
\end{pmatrix}
\,.
\end{equation*} 
Since the row vector only contains constants, 
%by Lemma~\ref{lem:Arc Growth} and its proof,
\begin{align*}
|V(\Xi^n)|
&\overset{n \to \infty}{\asymp}
\bigg\Vert \sum_{i=1}^{n-1} \boldsymbol{\chi}(\Xi^{i-1})\bigg\Vert_1 
\overset{n \to \infty}{\asymp}
\sum_{i=1}^{n-1} \rho(\mathbf{M})^i
\overset{n \to \infty}{\asymp}
\rho(\mathbf{M})^n\,.
\end{align*}

\end{proof}

\begin{theorem}
$\Xi$ is a sparse graph.
\end{theorem}
\begin{proof}
The graph density
\[
D(\Xi):=\lim_{n\to\infty}\frac{|E(\Xi^n)|}{|V(\Xi^n)|(|V(\Xi^n)|-1)}
=0\,.\qquad%\qedhere
\]
\end{proof}

%%%%%%%%%%%%%%%%%%%%%%%%%%%%%%%%%%%%%%%%%%%%
\section{Proof of Theorem~\ref{thm:Deterministic Substitution}}
\label{Section: Proof of fractality}
This section is devoted to proving Theorem~\ref{thm:Deterministic Substitution}.
However, before we are to do so, a few short but technical results are required, especially Theorem~\ref{thm:Spectral Combinatorial Matrix} and~\ref{thm:minimal Distance}.
\subsection{A combinatorial matrix problem}
\label{subSection: A combinatorial matrix problem}

\begin{lemma}\label{lem:Strict Collatz}
Suppose that, for any non-negative matrix $\mathbf{D}$ and positive row vector~$\mathbf{x}$,
$\min_i \frac{[\mathbf{D}\mathbf{x}]_i}{[\mathbf{x}]_i} < \max_i \frac{[\mathbf{D}\mathbf{x}]_i}{[\mathbf{x}]_i}$.
Then $\max_i \frac{[\mathbf{D}\mathbf{x}]_i}{[\mathbf{x}]_i} > \rho(\mathbf{D})$.
\end{lemma}
\begin{proof}
%By the Collatz-Wielandt Formula,
%$\max_i \frac{[\mathbf{D}\mathbf{x}]_i}{[\mathbf{x}]_i}\geq \rho(\mathbf{D})$.
%Therefore, assume that $\max_i \frac{[\mathbf{D}\mathbf{x}]_i}{[\mathbf{x}]_i} = \rho(\mathbf{D})$.
%Let $\mathbf{v}_\mathbf{D}$ be the left Perron-Frobenius vector of $\mathbf{D}$ as given by the Perron-Frobenius Theorem and set $\alpha = \mathbf{v}_\mathbf{D}\mathbf{x}$.
%Then
%\[
%    \rho(\mathbf{D})\alpha
%  = \rho(\mathbf{D})\mathbf{v}_\mathbf{D}\mathbf{x}
%  = \mathbf{v}_\mathbf{D}\mathbf{D}\mathbf{x}
%  = \sum_i [\mathbf{v}_\mathbf{D}]_i [\mathbf{D}\mathbf{x}]_i
%  = \sum_i \frac{[\mathbf{D}\mathbf{x}]_i}{[\mathbf{x}]_i} [\mathbf{v}_\mathbf{D}]_i [\mathbf{x}]_i\,.
%\]
%Since $\min_i \frac{[\mathbf{D}\mathbf{x}]_i}{[\mathbf{x}]_i} < \max_i \frac{[\mathbf{D}\mathbf{x}]_i}{[\mathbf{x}]_i}$,
%\[
%    \rho(\mathbf{D})\alpha
%  < \sum_i \max_i \frac{[\mathbf{D}\mathbf{v}]_i}{[\mathbf{x}]_i} [\mathbf{v}_\mathbf{D}]_i [\mathbf{x}]_i
%  = \sum_i \rho(\mathbf{D}) [\mathbf{v}_\mathbf{D}]_i[\mathbf{x}]_i
%  = \rho(\mathbf{D})\alpha\,,
%\]
%is a contradiction.
See Lemma~8.1 in \cite{Li2023}.
\end{proof}

Recall the defined notation from Definition~\ref{def:Path Matrix D}.

\begin{theorem}\label{thm:10.2}
Suppose that all matrices in $\mathcal{D}_{\min}$ are primitive,
and let $\mathbf{d}_i \in \mathcal{V}_i$
and $\mathbf{D}_{\min}\in\mathcal{D}_{\min}$.
Let $\mathbf{v}_{\min}$ be the corresponding right column Perron-Frobenius vector of $\mathbf{D}_{\min}$.
Then%, for all~$j$,
\[
  \mathbf{d}_i\mathbf{v}_{\min} \geq \rho(\mathbf{D}_{\min}) [\mathbf{v}_{\min}]_i\,.
\]
\end{theorem}

\begin{proof}
Assume that $\mathbf{d}_i\mathbf{v}_{\min} < \rho(\mathbf{D}_{\min}) [\mathbf{v}_{\min}]_i$. %for some $j$.
Define $\mathbf{D}$ to be the matrix obtained by replacing the $j$-th row of $\mathbf{D}_{\min}$ by $\mathbf{d}_i$,
and note that $\mathbf{D}\in \mathcal{D}$.
Then for each $j\neq i$,
\[
  [\mathbf{D}\mathbf{v}_{\min}]_j = [\mathbf{D}_{\min}\mathbf{v}_{\min}]_j = \rho(\mathbf{D}_{\min})[\mathbf{v}_{\min}]_j
\]
whereas
$[\mathbf{D}\mathbf{v}_{\min}]_i = \mathbf{d}_i\mathbf{v}_{\min} < \rho(\mathbf{D}_{\min}) [\mathbf{v}_{\min}]_i$.
Therefore,
\[
  \min_j \frac{[\mathbf{D}\mathbf{v}_{\min}]_j}{[\mathbf{v}_{\min}]_j}
  =\frac{[\mathbf{D}\mathbf{v}_{\min}]_i}{[\mathbf{v}_{\min}]_i}
  < \rho(\mathbf{D}_{\min})
  = \max_j \frac{[\mathbf{D}\mathbf{v}_{\min}]_j}{[\mathbf{v}_{\min}]_j}\,.
\]
By Lemma~\ref{lem:Strict Collatz},
\[
    \rho(\mathbf{D})
  < \max_j \frac{[\mathbf{D}\mathbf{v}_{\min}]_j}{[\mathbf{v}_{\min}]_j}
  = \rho(\mathbf{D}_{\min})\,,
\]
contradicting the definition of $\mathbf{D}_{\min}$.
\end{proof}

%\noindent
Next is our second auxiliary result.
\begin{theorem}\label{thm:Spectral Combinatorial Matrix}
For any number of matrices $\mathbf{D}_1,\ldots,\mathbf{D}_n\in\mathcal{D}$,
\[
  \rho\biggl(\prod_{\ell=1}^n \mathbf{D}_\ell\biggr) \geq \rho_{\min}(\mathcal{D})^n\,.
\]
\end{theorem}

\begin{proof}
For each $i$ and $\ell$, the row vector $[\mathbf{D}_\ell]_i$ satisfies
$[\mathbf{D}_\ell]_i\mathbf{v}_{\min} \geq\rho(\mathbf{D}_{\min})[\mathbf{v}_{\min}]_i$
by Theorem~\ref{thm:10.2}, so
\[
  \mathbf{D}_\ell\mathbf{v}_{\min} \geq\rho(\mathbf{D}_{\min})\mathbf{v}_{\min}\,.
\]
Then
\[
       \prod_{\ell=1}^n     \mathbf{D}_\ell \mathbf{v}_{\min}
  \geq \rho(\mathbf{D}_{\min})
       \prod_{\ell=1}^{n-1} \mathbf{D}_\ell \mathbf{v}_{\min}
  \geq \cdots
  \geq \rho(\mathbf{D}_{\min})^n\mathbf{v}_{\min},
\]
so
\[
  \min_i \frac{\displaystyle\biggl[\prod_{\ell=1}^n \mathbf{D}_\ell \mathbf{v}_{\min}\biggr]_i}{[\mathbf{v}_{\min}]_i} \geq \rho(\mathbf{D}_{\min})^n\,.
\]
By Collatz-Wielandt Formula,
\[
       \rho\biggl(\prod_{\ell=1}^n \mathbf{D}_\ell\biggr)
    =  \max_{\mathbf{x}>0} \min_i \frac{\displaystyle\biggl[\prod_{\ell=1}^n \mathbf{D}_\ell \mathbf{x}       \biggr]_i}{[\mathbf{x}       ]_i}
  \geq \min_i \frac{\displaystyle\biggl[\prod_{\ell=1}^n \mathbf{D}_\ell \mathbf{v}_{\min}\biggr]_i}{[\mathbf{v}_{\min}]_i}
  \geq \rho_{\min}(\mathcal{D})^n\,.
\]
\end{proof}
\begin{remark}
    If $\mathcal{D}$ is not defined the way as Definition~\ref{def:Path Matrix D}, then the theorem does not hold.
    A straight example is given by 
    \[
    \mathcal{D} 
    =\big\{
    \mathbf{D}_1= 
    \begin{pmatrix}   
    1 & 1 \\
    1 & 2
    \end{pmatrix},
    \mathbf{D}_2= 
    \begin{pmatrix}   
    2 & 1 \\
    1 & 1
    \end{pmatrix}
    \big\} 
    \]
    with $\rho(\mathbf{D}_1)=\rho(\mathbf{D}_2)\approx2.6180$.
    Then $5.8284\approx\rho(\mathbf{D}_1\mathbf{D}_2)\leq\rho(\mathbf{D}_1)^2=\rho(\mathbf{D}_2)^2\approx6.8539$.

    However,
    \[
    \mathcal{D} 
    =\big\{
    \mathbf{D}_1= 
    \begin{pmatrix}   
    1 & 1 \\
    1 & 2
    \end{pmatrix},
    \mathbf{D}_2= 
    \begin{pmatrix}   
    2 & 1 \\
    1 & 1
    \end{pmatrix},
    \mathbf{D}_3= 
    \begin{pmatrix}   
    2 & 1 \\
    1 & 2
    \end{pmatrix},
    \mathbf{D}_4= 
    \begin{pmatrix}   
    1 & 1 \\
    1 & 1
    \end{pmatrix}
    \big\} 
    \]
    is the family satisfying the Theorem~\ref{thm:Spectral Combinatorial Matrix}.
    The inherent property of this peculiar matrix family {\bf fundamentally ensures the existence of $\Xi$}.
\end{remark}

%%%%%%%%%%%%%%%%%%%%%%%%%%%%%%%%%%%%%%%%%%%
\subsection{Distance growth between nodes}
\label{subSection: Estimation of distance}
For a path $P$ in a graph, let $|P|$ denote the {\em length} of $P$; that is, the number of arcs (edges) in~$P$.
The {\em (minimal) distance} $d(u,v)$ between any two nodes $u$ and $v$ in a graph
is the shortest length $|P|$ of any path between $u$ and $v$.
Here, we adopt the convention that {\em paths} visit any node at most once and do therefore not contain cycles or loops.

In~\cite{LiYaWa19},
the authors provided an analysis of the distances between nodes
in deterministic iterated graph systems,
and it was proved that these distances grew with bounded growth rates.
However, a feasible algorithm was not given to calculate these rates;
they are indeed difficult to determine,
and it becomes more complicated to do so for random iterated graph systems.

To illustrate why these distances are difficult to estimate,
consider Fig.~\ref{fig:example3} in which a deterministic iterated graph system is shown,
with just two substitution rules: one for blue arcs and one for red arcs.
In the shown substitution sequence $\Xi^0\to \Xi^1\to \Xi^2$,
the distance between the initial nodes $A$ and $B$ is~1.
In $\Xi^1$, this distance is~2 and is the length of the path $P^{(1)}$ between $A$ and $B$.
However,
substituting the arcs of $P^{1}$ gives paths of length~10 between $A$ and $B$;
this is larger than 9, the distance between $A$ and $B$.
Rather, the shortest paths between $A$ and $B$, such as $P^{(2)}$,
here arise disjointly and independently from the arcs of $P^{1}$.

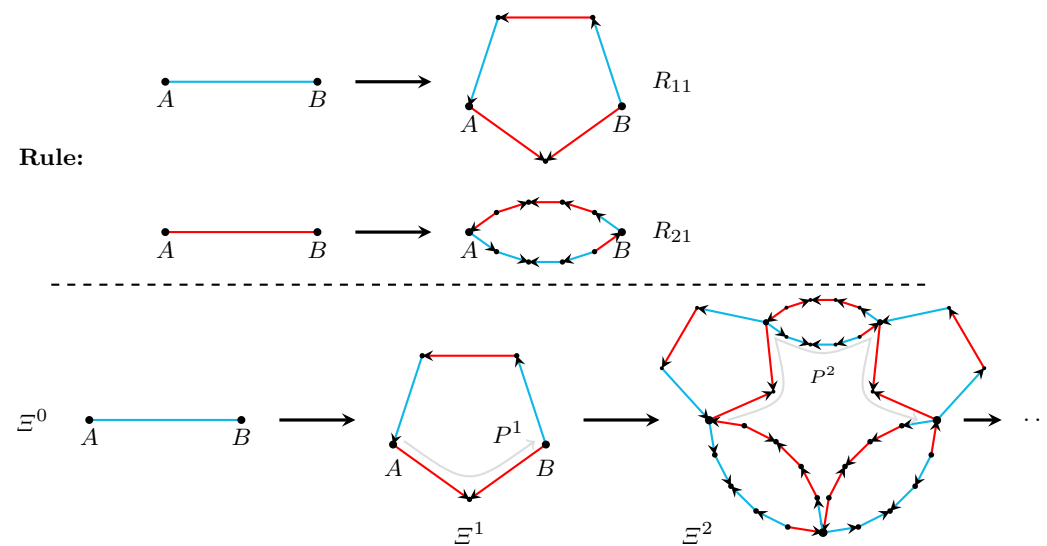
\begin{figure}[ht]
\centering
\begin{tikzpicture}[scale=1,black,thick]
  \begin{scope}[scale=1,shift={(1,2)}]
  \draw (-1.5,-1) node {\normalsize \textbf{Rule:}};
    \draw[draw=myaqua!90!black] (0,0) -- (2,0);
    \fill (0,0) circle (1.5pt) (2,0) circle (1.5pt);
    \draw (0,0) node[below] {\footnotesize $A$};
    \draw (2,0) node[below] {\footnotesize $B$};
    \draw[very thick,>=stealth,->] (2.5,0) --++ (1,0);
    \Rrulea{5,0}{.67}{0}{\draw (-162:1.577) node[below] {\footnotesize $A$};
                         \draw (-18: 1.577) node[below] {\footnotesize $B$};
                         \draw (2.5,0) node {\makebox(0,0){$R_{11}$}};}
  \end{scope}
  \begin{scope}[scale=1,shift={(1,0)}]
    \draw[draw=red] (0,0) -- (2,0);
    \fill (0,0) circle (1.5pt) (2,0) circle (1.5pt);
    \draw (0,0) node[below] {\footnotesize $A$};
    \draw (2,0) node[below] {\footnotesize $B$};
    \draw[very thick,>=stealth,->] (2.5,0) --++ (1,0);
    \Rruleb{5,0}{.67}{0}{\draw (-1.5,0) node[below] {\footnotesize $A$};
                         \draw ( 1.5,0) node[below] {\footnotesize $B$};
                         \draw ( 2.5,0) node {\makebox(0,0){$R_{21}$}};}
    \draw[dashed] (-1.5,-.7) -- (10,-.7);
  \end{scope}
  \begin{scope}[scale=1,shift={(0,-2.5)}]
    \draw[draw=myaqua!90!black] (0,0) -- (2,0);
    \fill (0,0) circle (1.5pt) (2,0) circle (1.5pt);
    \draw (0,0) node[below] {\footnotesize $A$};
    \draw (2,0) node[below] {\footnotesize $B$};
    \draw (-.75,0) node {$\Xi^0$};
    \draw[very thick,>=stealth,->] (2.5,0) --++ (1,0);
    \Rrulea{5,0}{.67}{0}{\draw (342:.8) node(){{\footnotesize$P^{1}$}};
                         \draw[gray!75!white,->] (198:1.35) .. controls (270:1.35) .. (342:1.35);}
    \draw (4,-.4) node[below] {\footnotesize $A$};
    \draw (6,-.4) node[below] {\footnotesize $B$};
    \draw (5,-1.5) node {$\Xi^1$};
    \draw[very thick,>=stealth,->] (6.5,0) --++ (1,0);
    \begin{scope}[shift={(9.65,0)},scale=.5]
      \Rrulea{150:3.085}{1}{ 60}{}
      \Rrulea{ 30:3.085}{1}{300}{}
      \Rruleb{  0,2.6  }{1}{  0}{}
      \Rruleb{ 1.5,-1.5}{1.41}{45}{}
      \Rruleb{-1.5,-1.5}{1.41}{315}{}
      \fill (-3,0) circle (2.25pt) (3,0) circle (2.25pt);
      \draw[gray!75!white,->] (180:2.5) .. controls (150:1) .. (120:2.5) ..
        controls (90:1.65) .. ( 60:2.5) .. controls ( 30:1) .. (0:2.5);
      \draw ( 90:1.2) node(){{\scriptsize$P^{2}$}};
      
    \end{scope}
    %\draw (8    ,0) node[below] {\footnotesize $A$};
    %\draw (11.25,0) node[below] {\footnotesize $B$};
    \draw (8,-1.5) node {$\Xi^2$};
    \draw[very thick,>=stealth,->] (11.5,0) --++ (0.5,0);
    \draw (12.5,0) node {$\dots$};
  \end{scope}
\end{tikzpicture}
\caption{An example of deterministic iterated graph system}
\label{fig:example3}
\end{figure}

In general, the shortest paths between nodes in each graph $\Xi^n$
and the shortest paths between nodes in $\Xi^{(n+1)}$
may be greatly independent of each other.
Nevertheless, it is possible to investigate how distances develop during the substitution process.
For this purpose, consider the two nodes $A$ and $B$ in $\Xi^0$
and let $A$ and $B$ also denote these two fixed nodes in each subsequent graph $\Xi^n$,
as indicated in Fig.~\ref{fig:example3}.
Let $\Gamma^n$ denote a shortest path between~$A$ and~$B$ \emph{}in $\Xi^{n}$.

\bigskip

\begin{definition}\label{def:Uniform}
{\rm
A path $P^{n+1}$ in $\Xi^{(n+1)}$ between $A$ and $B$ is {\em uniform} if
it arises by successive edge substitutions from paths $P^0,\ldots,P^n$ in $\Xi^{0},\ldots,\Xi^{n}$, respectively
such that, for each $s = 0,\ldots,n$ and each colour~$i$,
all edges of colour $i$ in $P^s$ when substituted are in $P^{s+1}$ replaced by
the same path between $A$ to $B$ in each $R_{i}$ for all $i$.
For instance, the path $P^{2}$ in Fig.~\ref{fig:example3} is uniform.}
\end{definition}

\begin{lemma}\label{lem:Uniform}
For each $n$, at least one shortest path $\Gamma^n$ is uniform.
\end{lemma}
\begin{proof}
For $n = 1$, the lemma is trivially true,
so assume that some shortest path $\Gamma^n$ in $\Xi^{n}$ is uniform.
Since $\Xi^{n+1}$ can be obtained by substituting every $i$-coloured arc in $\Xi^{1}$ by some $\Xi^n$,
each shortest path $\Gamma^{n+1}$ in $\Xi^{n+1}$
is obtained by concatenating a sequence of paths $\Gamma^n$ in different $\Xi^{n}$.
By choosing the same path $\Gamma^n$ for each $i$-coloured arc,
the resulting shortest path $\Gamma^{n+1}$ will be uniform.
Induction completes the proof.
\end{proof}

\bigskip

Now we focus on calculating the lengths of uniform paths recursively.
Recall from Definition~\ref{def:Path Matrix D} that
$\mathcal{P}_{i}$ is the set of subgraphs of $R_{i}$
whose underlying (undirected) graph is a path between~$A$ and~$B$.
Each element $\mathcal{C}= (P_{1},\ldots,P_{\lambda})\in\mathcal{P}$
is of the Cartesian product 
\[
  \mathcal{P} = \prod_{j=1}^{\lambda} \mathcal{P}_{i}\,.
\]
%Define
%\[
%  \boldsymbol{\nu}(\mathcal{C}_i) = \sum_{k=1}^{q_i} p_{ik} \boldsymbol{\chi}(P_{ik})\,.
%\]

Note that if $P^n$ is a uniform path between~$A$ and~$B$,
then the subsequent uniform paths $P^{n+1}$ arising from $P^n$
are obtained by choosing an element $\mathcal{C}  \in\mathcal{P}$
and replacing each $i$-coloured edge of $P^{n}$ by the value of each element in $\mathcal{C}$. 
For each tuple
\[
  \mathcal{C} = (P_1, \ldots,P_\lambda )\in\prod_{i=1}^\lambda \mathcal{P}_i\,,
\]
define $\mathbf{D}_{\mathcal{C}}$ to be the matrix
\[
 \mathbf{D}_{\mathcal{C}}=
 \begin{pmatrix}\boldsymbol{\chi}(\mathcal{C}_1)\\[-1mm]\vdots\\\boldsymbol{\chi}(\mathcal{C}_\lambda)\end{pmatrix}\,.
\]
Note that $\mathbf{D}_{\mathcal{C}} \in \mathcal{D}$ by Definition~\ref{def:Path Matrix D}.
\begin{lemma}\label{lem:10.7}
\[ 
  |\Gamma^{n}|=\min_{\mathbf{D}_{\mathcal{C}} \in \mathcal{D}}\Vert \boldsymbol{\chi}(\Xi^0)\prod_{l=1}^{n} \mathbf{D}_{\mathcal{C}}\Vert_1\,.
\]
\end{lemma}
\begin{proof}
Each given tuple $\mathcal{C}= (P_1, \ldots,P_\lambda )$ determines a unique expected uniform path $P^{n+1}$ based on the given uniform path $P^n$.
To determine the length of $P^{n+1}$ explicitly,
recall the notation from Definition~\ref{def:Path Matrix D}.
In particular, 
$|P^{n+1}| = \Vert\mathbf{\boldsymbol{\chi}}(P^{n+1})\Vert_1$
and
\[
    \Vert \mathbf{\boldsymbol{\chi}}(P^{n+1})\Vert_1 
    = \sum_{i=1}^\lambda    \mathbf{\boldsymbol{\chi}}(P^n) \boldsymbol{\chi}(P_{i})^T
    = \Vert \mathbf{\boldsymbol{\chi}}(P^n)\mathbf{D}_{\mathcal{C}}\Vert_1\,.
\]
%where $\mathbf{D}_{\mathcal{P}}\in\mathcal{D}$ is the matrix whose $i$-th rows are
%\[
%  \mathbf{d}_{i} = \boldsymbol{\nu}(\boldsymbol{P}_i)\,.
%\]
If we fix $\Xi^0$,
then
\[
     \Vert \mathbf{\boldsymbol{\chi}}(P^{n+1})\Vert_1 
  = \Vert \mathbf{\boldsymbol{\chi}}(\Xi^0)\prod_{l=1}^{n} \mathbf{D}_{\mathcal{C}}\Vert_1\,.
\]
As at least one shortest path $\Gamma^n$ is proved to be uniform,
\[ 
  |\Gamma^{n}| 
  = \min_{\mathbf{D}_{\mathcal{C}} \in \mathcal{D}}\Big\Vert \boldsymbol{\chi}(\Xi^0)\prod_{l=1}^{n} \mathbf{D}_{\mathcal{C}}\Big\Vert_1\,.\qquad%\qedhere
\]
\end{proof}

\begin{theorem}\label{thm:minimal Distance}
\[
| \Gamma^n | \overset{n \to \infty}{\asymp} \rho_{\min}(\mathcal{D})^n \,.
\]
\end{theorem}
\begin{proof}
By Lemma~\ref{lem:10.7}, Lemma~\ref{lem:Norm Range} and Theorem~\ref{thm:Spectral Combinatorial Matrix},
\begin{align*}
     | \Gamma^{n} |  
     &= \min_{\mathbf{D}_{\mathcal{C}} \in \mathcal{D}}\Vert \boldsymbol{\chi}(\Xi^0)\prod_{l=1}^{n} \mathbf{D}_{\mathcal{C}} \Vert_1 \\
    &\overset{n \to \infty}{\asymp} \min_{\mathbf{D}_{\mathcal{C}}\in \mathcal{D}}  
        \rho\Bigl( \prod_{\ell=1}^{n} \mathbf{D}_{\mathcal{C}}\Bigr) 
    = \min_{{\mathbf{D}_\ell} \in \mathcal{D}}\rho\Bigl(\prod_{\ell=1}^n \mathbf{D}_\ell\Bigr) 
     = \rho_{\min}(\mathcal{D})^n \,.
\end{align*}
\end{proof}

%%%%%%%%%%%%%%%%%%%%%%%%%%%%%%%%%%%%%%%%%%%
\subsection{Proof of Theorem~\ref{thm:part I}}
\label{Subsection: Proof of fractality}

\begin{definition}{\rm
Let $N^*_\ell(\hat{G})$ be the maximum number of node-disjoint boxes of $\hat{G}$ with diameter~$\ell$. 
That is, $N^*_\ell(\hat{G})$ is the packing number.}
\end{definition}
\begin{lemma}\label{lem:Packing is Less Than Covering}
  $N^*_\ell(\hat{G})\leq N_\ell(\hat{G})$.
\end{lemma}

\begin{proof}
Choose $N^*_\ell(\hat{G})$ nodes whose pairwise distances are greater than $\ell$.
Since no subgraph of diameter at most~$\ell$ contains more than one of the covering boxes, 
$N_\ell(\hat{G})\geq N^*_\ell(\hat{G})$.
\end{proof}

\begin{notation}
Let $f(x) \succeq g(x)$ denotes $f(x)\geq c \cdot g(x)$ for some given constant $c>0$. Similar for $f(x) \preceq g(x)$.
Recall that $\Delta(G)$ denotes the diameter of graph $G$.
\end{notation}

\begin{lemma}
\label{lem:Diameter}
  \[
  \Delta(\Xi^n)\overset{n \to \infty}{\asymp} \rho_{\min}(\mathcal{D})^n\,.
  \]
\end{lemma}
\begin{proof}
Let $\Delta_{\Xi^n}(V(\Xi^{n-k}))$ denote 
\[
\sup \{\dist_{\Xi^n}(u,v)\,:\, u\in V(\Xi^{n-k}),v\in V(\Xi^{n-k})\}\,.
\]
Note that if $k_1\geq k_2$,
then $V(\Xi^{n-k_1}) \subset V(\Xi^{n-k_2})$, so 
\[
  \Delta_{\Xi^n}(V(\Xi^{n-k_1})) \leq \Delta_{\Xi^n}(V(\Xi^{n-k_2}))\,.
\]
Observe that
\begin{align*}
&\quad\,\Delta_{\Xi^n}(V(\Xi^{n-k}))\\
&\leq \Delta_{\Xi^n}(V(\Xi^{n-k-1}))+2\max\{ \dist_{\Xi^n}(u,v)\::\: u\in V(\Xi^{n-k-1}),v\in V^*(\Xi^{n-k}) \} \\    
&\leq \Delta_{\Xi^n}(V(\Xi^{n-k-1}))+2\max_{i}\Delta(R_{i})|\Gamma^{k}| \,.
\end{align*}
Let $R=\max_{i}\Delta(R_{i})$.
Then
\begin{align*}
\Delta(\Xi^n) &\leq \Delta_{\Xi^n}(V(\Xi^{n-1})) +2R \\
&\leq \Delta_{\Xi^n}(V(\Xi^{n-2}))+2R|\Gamma^1|+2R  \\
&\leq \cdots \leq \Delta_{\Xi^n}(V(\Xi^{0}))+2R|\Gamma^{n-1}|+\cdots+2R|\Gamma^1|+2R \,.
\end{align*}
Note that here, $\Gamma^n$ in fact refer to various shortest path given different initial graphs.
However, Theorem~\ref{thm:minimal Distance} sheds light on this problem: all $\Gamma^n$ are asymptotically equivalent without regard to some constant times.
Hence, we say the right-hand side above
\begin{align*}
\text{RHS} 
%&\leq \dist_{\Xi^n}(A,B)+2R\sum_{j=0}^{n-1}|\Gamma^{j}|  \\
&\leq \max\{|\Gamma^n|\}+2R\sum_{j=0}^{n-1}\max\{|\Gamma^{j}|\}  \\
&\leq 2R \sum_{j=0}^{n}\max\{|\Gamma^{j}|\} \\
& \preceq \sum_{j=0}^{n}\max\{|\Gamma^{j}|\}\,.
\end{align*}
By Lemma~\ref{lem:Norm Range}, 
$\rho(\mathbf{D}_{\min})^n \preceq \Vert \mathbf{x}\mathbf{D}^n_{\min}\Vert_1 \preceq \rho(\mathbf{D}_{\min})^n$, as a result
\begin{align*}
      \rho_{\min}(\mathcal{D})^n
 \preceq |\Gamma^n| 
&\leq \Delta(\Xi^n) \\
&\preceq \sum_{j=0}^{n}\max\{|\Gamma^{j}|\}
\preceq \sum_{j=0}^{n}\rho_{\min}(\mathcal{D})^j
\preceq\rho_{\min}(\mathcal{D})^n\,.
\end{align*} \qedhere

\end{proof}

\noindent
%Recall Theorem~\ref{thm:Deterministic Substitution}:

%\begin{theorem}
%If $\mathbf{M}$ and $\mathbf{D}_{\min}$ are primitive (or positive) and invertible, or are $2 \times 2$ primitive matrices,
%then almost every graph $\Xi \in \mathcal{G}^\infty$ has the fractality with
%\[
%    \dim_B(\Xi)
% = \lim_{\ell \to 0}\frac{\log N_\ell(\hat{G}^\infty)}{-\log \ell}
%  = \dfrac{\log\rho(\mathbf{M})}{\log \rho_{\min}(\mathcal{D})}\,.
%\]
%\end{theorem}
\bigskip
We now prove Theorem~\ref{thm:part I}.
\begin{proof}
By Lemma~\ref{lem:Packing is Less Than Covering},
\[
N_L^* (\Xi^n) \leq N_L (\Xi^n)\,.
\]
Set $L=\rho_{\min}(\mathcal{D})^{n-k}$.
We can easily find a finite constant $C$ depending on random iterated graph systems only such that $N_L^*(\Xi^n) \succeq |E(\Xi^{k})|$ in that every arc in $E(\Xi^{n-k})$ turns into a $\Xi^{k}$ in $\Xi^n$;
Also, $N_L(\Xi^n) \preceq |E(\Xi^{k})|$ since constant number of subgraphs with diameter $L$ can cover $\Xi^{n-k}$.
Therefore,
\[
|E(\Xi^k)|\preceq N_L^* (\Xi^n) \leq N_L (\Xi^n) \preceq |E(\Xi^k)|\,.
\]
Recall Theorem~\ref{lem:Arc Growth}, providing that
\[
N_L(\Xi^n) 
\overset{n \to \infty}{\asymp}
|E(\Xi^k)|
\overset{n \to \infty}{\asymp}
\rho(\mathbf{M})^{k} \,.
\]
Then 
\begin{align*}
\dim_B (\Xi)&=
\lim_{\ell \to 0} \frac{\log N_\ell(\hat{\Xi}^n)}{-\log \ell} \\
&=\lim_{n \to \infty} \frac{\log N_L(\Xi^n)}{-\log \frac{L}{\Delta(\Xi^n)}} \\
&=\lim_{n \to \infty} \frac{\log \rho(\mathbf{M})^k}{-\log \frac{\rho_{\min}(\mathcal{D})^{n-k}}{ \rho_{\min}(\mathcal{D})^n}}
=\dfrac{\log\rho(\mathbf{M})}{\log \rho_{\min}(\mathcal{D})}\,.
\end{align*}
\end{proof}

% \bigskip

%As a corollary, deterministic substitution networks are particular random substitution networks where $q_i=1$ and $p_{ij}=1$.  
%The only distinction is that Theorem~\ref{thm:Stochastic Substitution System} is not used for deterministic cases.
%Hence under this situation, $\mathbf{M}$ and all $\mathbf{D}_{\min}$ are no longer required to be invertible.

%%%%%%%%%%%%%%%%%%%%%%%%%%%%%%%%%%%%%%%%%%%
\subsection{Existence for the limited graphs of deterministic iterated graph systems}
\label{subsection: Existence}
This section aims to prove that Hausdorff dimension is consistent with Minkowski dimension for random iterated graph systems.

%\subsection{Gromov–Hausdorff Convergence}
Let $\dist^H(X,Y)$ be the Hausdorff distance between two subsets $X$ and $Y$ on a metric space~\cite{Hau20}.
For any two metrics spaces $X$ and $Y$, define 
\[
\dist^{GH}=\inf\{\dist_\varrho^H (X,Y) \::\: \varrho \text{ is a admissible metric on } X\cup Y \}
\]
be the Gromov–Hausdorff distance in~\cite{Gro07}.
\begin{theorem}
$\Xi$ is the Gromov–Hausdorff limit of substitution networks sequence.
That is,
\[
\lim_{n \to \infty} \dist^{GH}(\hat{\Xi}^n,\hat{\Xi})=0\,.
\]
\end{theorem}
\begin{proof}
First we show the existence
\begin{equation}\label{eq:distance}
\lim_{n\to \infty} \dist_{\hat{\Xi}^n}(v_1,v_2) = c
\end{equation}
where $c\in [0,1]$ for any $v_1\in V(\Xi^{k_1})\subset V(\Xi^n)$ and $v_2\in V(\Xi^{k_2})\subset V(\Xi^n)$.
This is essentially because that $\Gamma^n \asymp \Delta(\Xi^n) \asymp \rho_{\min}(\mathcal{D})^n$, otherwise there won't be a limit.
We thus have, for a constant $c>0$
\begin{align*}
\dist_{\hat{\Xi}^n}(v_1,v_2) 
&\overset{n \to \infty}{\asymp} 
\frac{\diam(G^{k_2})|\Gamma^{n-k_2}|}{\diam(\Xi^n)} \\
&\overset{n \to \infty}{\asymp} 
\frac{\rho_{\min}(\mathcal{D})^{k_2}\rho_{\min}(\mathcal{D})^{n-k_2}}{ \rho_{\min}(\mathcal{D})^n}
=1\,.
\end{align*}
As the distance in any $\hat{\Xi}$ is always less or equal to $1$, we obtain the result above by $\displaystyle\lim_{n\to \infty} \dist_{\hat{\Xi}^n}(v_1,v_2)=c\leq 1$.

$\hat{\Xi}$ naturally constructs a metric space $\bigl(V(\hat{\Xi}),\dist_{\hat{\Xi}}\bigr)$.
%Hence we define for any positive $\varepsilon$
%\[
%\tilde{\dist}^{GH} \bigl((X,x),(Y,y)\bigr)
%=\dist^{GH}\bigl(B_X(x,R),B_Y(y,R)\bigr)
%\]
%for any $x \in X$ and $y \in Y$.
%By~\cite{Be96} 
A sequence of pointed metric spaces $(X_n,x_n)$ is said to converge to $(Y,y)$ if for any $\varepsilon>0$
\[
\lim_{n\to \infty} \dist^{GH}\bigl(B_{X_n}(x_n,\varepsilon),B_Y(y,\varepsilon)\bigr)=0\,.
\]

Note that $V(\Xi)=\bigcup_{i=0}^\infty V(\Xi^n)$ where $V(\Xi^0)\subset V(\Xi^1) \subset \cdots$ is a nested set.
We then claim that for any $v\in V(\Xi^n)$
\[
\lim_{n\to \infty} \dist^{GH}\bigl(B_{\hat{\Xi}^n}(v,\varepsilon),B_{\hat{\Xi}}(v,\varepsilon)\bigr)=0 \,.
\]
This is because by Equation (\ref{eq:distance}) all distances between nodes in the box $B_{\hat{\Xi}^n}(v,\varepsilon)$ converge. 
In other words, $B_{\hat{\Xi}^n}(v,\varepsilon)$ and $B_{\hat{\Xi}}(v,\varepsilon)$ are isometric when $n$ goes to infinity under some admissible metric related to the graph distance.

Therefore every node $v\in V(\hat{\Xi}^n)$ is Gromov–Hausdorff convergent to $v$ in $\hat{\Xi}$. 
Finally provided that every $\hat{\Xi}^n$ is compact (finite graphs) with uniformly bounded diameter $1$, the desired result is obtained.

\end{proof}

\subsection{Proof of Theorem~\ref{thm:Part II}}\label{Subsection:Proof of Hausdorff}
This subsection aims to demonstrate a result for the Hausdorff dimension of $\Xi$.
It can be done by showing $\Xi$ is quasi self-similar and then applying Falconer's implicit theorem~\cite{falconer1989dimensions}.
However, since this subsection will define a slightly different Hausdorff dimension for graphs, we are to present a full proof for the desired result.

\cite{Fed14,Rog98} have sufficiently investigated Hausdorff measure and dimension on metric spaces.
We now introduce a definition of Hausdorff dimension for graphs.
\begin{definition}
Define Vertex Hausdorff measure for graph $G$
\[
\mathcal{HV}^s(G)= \lim_{\delta \to 0}\inf \left\{ \sum_{i=1}^{\infty} \Delta(U_i)^s:\{U_i\} \text{ is a } \delta_\mathcal{V}\text{-cover of } V(\hat{G}) \right\}\,,
\]
Then set $\dim_{HV} (G)=\inf\{s\geq 0:\mathcal{HV}^s(G)=0\}$.
Here $\{U_i\}$ is a $\delta_\mathcal{V}$-cover for $V(\hat{G})$ if 
\[
V(\hat{G})\subset\bigcup_{i=1}^\infty V(U_i)\,.
\]

\smallskip

Define Edge Hausdorff measure for graphs that
\[
\mathcal{HE}^s(G)=\lim_{\delta \to 0}\inf \left\{ \sum_{i=1}^{\infty} \Delta(U_i)^s:\{U_i\} \text{ is a } \delta_\mathcal{E}\text{-cover of } E(\hat{G}) \right\}\,.
\]
Then set $\dim_{HE}(G)=\inf\{s\geq 0:\mathcal{HE}^s(G)=0\}$.
Here $\{U_i\}$ is a $\delta_\mathcal{E}$-cover for $E(\hat{G})$ if 
\[
E(\hat{G})\subset\bigcup_{i=1}^\infty E(U_i)\,.
\]

Let the Hausdorff dimension for graphs be
\[
\dim_H(G)=\max\Bigl\{\dim_H \bigl(V(G)\bigr),\dim_H \bigl(E(G)\bigr) \Bigr\} \,.
\]
\end{definition}

\begin{remark}\label{remark:Hausdorff}{\rm
There is no inequality between $\dim_H \bigl(V(G)\bigr)$ and $\dim_H \bigl(E(G)\bigr)$.

An example is $\Xi$.
Choosing any $n$, there is always $|(V(\Xi^n)|\geq|(E(\Xi^n)|$;
However, $V(\Xi)$ is countable while $E(\Xi)$ is uncountable.
Hence, we have $\dim_{HV}(\Xi)=0$, but uncountably many arcs imply that $\dim_{HE}(\Xi)>0$.
Though this is very much a straight fact, it is occasionally called the paradox of the binary tree.}
\end{remark}

\begin{theorem}
\[
  \dim_H (\Xi) = \dim_B (\Xi) = \frac{\log \rho(\mathbf{M})}{\log\rho_{\min}(\mathcal{D})}\,.
\]
\end{theorem}
\begin{proof}
By Theorem~\ref{thm:part I},
\[
  \dim_B(\Xi)=\frac{\log \rho(\mathbf{M})}{\log\rho_{\min}(\mathcal{D})}\,.
\]
By Remark~\ref{remark:Hausdorff}
\[
\dim_H(\Xi)=\dim_{HE}(\Xi)\,.
\]
It is trivial to check that
\[
  \dim_H(\Xi) \leq \dim_B(\Xi)\,.
\]
Hence, we only need to calculate the lower bound.

Fix $n$ and consider $\Xi^n$. 
Let $\{U_i\}$ be a $\delta$-cover of $E(\hat{\Xi}^n)$.
Then set
\[
k=\inf\{ s \::\: \max \{|\Gamma^s|\}\geq \Delta(U_i)\Delta(\Xi^n)  \}\,.
\]
Accordingly, as Theorem~\ref{thm:minimal Distance} points out the growth of $|\Gamma^k|$, we can also have
\[
\max \{|\Gamma^k|\}\preceq \Delta(U_i)\Delta(\Xi^n)\,.
\]
Picking any node $v_{U_i}\in U_i$, let
\[
\tau_i=\bigl\{ v\::\:\dist_{\Xi^n}(v,v_{U_i})\leq \max\{|\Gamma^k|\} \bigr\} \,.
\]
Then it is clear that $U_i\subset \tau_i$ and thus $|V(U_i)| \leq |V(\tau_i)|$.
Given that $\Delta(\tau_i)\Delta(\Xi^n)\leq 2 \max\{|\Gamma^k|\}$, we can always find a constant $c$ (not depending on $n$) such that $\tau_i$ is a subgraph of some $\Xi^{k+c}$.
This is because substitution networks can be reckoned as a composition of many $\Xi^{k+c}$ (if $n$ is large enough). 

As a result,
\[
U_i \subset \tau_i \subset \Xi^{k+c}\subset \Xi^n
\]
and hence $|V(U_i)| \leq |V(\Xi^{k+c})|$.
Therefore,
\begin{align*}
\bigl(\Delta(U_i)\Delta(\Xi^n)\bigr)^{\frac{\log \rho(\mathbf{M})}{\log\rho_{\min}(\mathcal{D})}} 
&\succeq \max\{|\Gamma^k|\}^{\frac{\log \rho(\mathbf{M})}{\log\rho_{\min}(\mathcal{D})}} \\ 
&\succeq \Delta(\Xi^k)^{\frac{\log \rho(\mathbf{M})}{\log\rho_{\min}(\mathcal{D})}} \\
&\succeq \Delta(\Xi^{k+c})^{\frac{\log \rho(\mathbf{M})}{\log\rho_{\min}(\mathcal{D})}} \\
&\succeq |V(\Xi^{k+c})| 
\geq |V(U_i)|
\end{align*}
by Lemmas~\ref{thm:minimal Distance} and Lemma~\ref{lem:Diameter}.

Finally, because $\{U_i\}$ is a covering, by Lemma~\ref{lem:Node Growth},
\begin{align*}
\mathcal{HE}^{\frac{\log \rho(\mathbf{M})}{\log\rho_{\min}(\mathcal{D})}}(\hat{\Xi}) 
&= \lim_{n\to \infty}\sum_{i=1}^{\infty} \Delta(U_i)^{\frac{\log \rho(\mathbf{M})}{\log\rho_{\min}(\mathcal{D})}} \\
&=\lim_{n\to \infty}\sum_{i=1}^{\infty} \bigg(\frac{\Delta(U_i)\Delta(\Xi^n)}{\Delta(\Xi^n)}\bigg)^{\frac{\log \rho(\mathbf{M})}{\log\rho_{\min}(\mathcal{D})}}\\
&\succeq\lim_{n\to \infty}\frac{\sum_{i=1}^\infty |V(U_i)|}{\Delta(\Xi^n)^{\frac{\log \rho(\mathbf{M})}{\log\rho_{\min}(\mathcal{D})}}}\\
&\succeq\lim_{n\to \infty} \frac{ |V(\Xi^n)|}{\Delta(\Xi^n)^{\frac{\log \rho(\mathbf{M})}{\log\rho_{\min}(\mathcal{D})}}}
\geq\lim_{n\to \infty} \frac{\rho(\mathbf{M})^n}{\rho(\mathbf{M})^n} 
=1\,.
\end{align*}

%The above implies that $\dim_H(\Xi)\geq \frac{\log \rho(\mathbf{M})}{\log\rho_{\min}(\mathcal{D})}$.
%Therefore,
%\[
%\frac{\log \rho(\mathbf{M})}{\log\rho_{\min}(\mathcal{D})} \leq 
%\dim_H(\Xi) \leq \dim_B(\Xi)
%=\frac{\log \rho(\mathbf{M})}{\log\rho_{\min}(\mathcal{D})} \,.
%\]
Hence,
\[
\dim_H(\Xi) = \dim_B(\Xi)=\frac{\log \rho(\mathbf{M})}{\log\rho_{\min}(\mathcal{D})},
\]
which completes the proof of Theorem~\ref{thm:Deterministic Substitution}.
\end{proof}

%%%%%%%%%%%%%%%%%%%%%%%%%%%%%%%%%%%%%%%%%%%%%
\section{Random iterated graph systems and  proof of Theorem~\ref{thm:Random Substitution}}\label{sec:Random substitution}

\subsection{Random iterated graph systems}\label{sec:random iterated graph systems}
\begin{definition}
A triple $\mathscr{R}=(\Xi^0,\mathcal{R},\mathcal{Q})$ is called a {\em random iterated graph system} if
\begin{gather*}
\Xi^0 \text{ is a finite directed graph,} \\
\mathcal{R}=\{\{R_{ij}\}_{j=1}^{q_i}\}_{i=1}^{\lambda} \text{ is a family of directed graphs where } q_i,\lambda\in\mathbb{N} \\ 
\text{and } \mathcal{Q}=\{(p_{i1},\dots,p_{iq_i})\}_{i=1}^{\lambda} \text{ is a set of probability vectors.} 
\end{gather*}
\end{definition}
See Fig.~\ref{fig:rule-example2} for example.
Recall the number of colours is $\lambda \in \mathbb{N}$.
For each $i=1,\ldots,\lambda$, set positive integers $q_i$, which is the number of possible rule graphs of colour $i$. 
Each directed rule graph $R_{ij}$ has a node $A$ and a node $B$ that respectively replace the beginning node $A$ and ending node $B$ of $e$;
this will determine exactly how $R_{ij}$ replaces~$e$.
In addition, in this paper we always require $\dist_{R_{ij}}(A,B)\geq 2$ for all $i$ and $j$.

To construct a sequence of growing graphs, start from the initial graph $\Xi^0$, which is usually just a single arc.
We construct $\Xi^1$ by replacing all $k$-coloured arc in $\Xi^0$ randomly by a graph in $\mathcal{R}$ according to some probability in $\mathcal{Q}$.
Note here, as stated above, the substitutions regarding node $A$ and $B$ are unique.
We then iteratively replace all $k$-coloured arcs for all $k$ in $\Xi^1$ to obtain $\Xi^2$.
In this way, denote the graph after $t\in \mathbb{N}$ iterations by $\Xi^n$;
accordingly, we have a sequence of graphs $\{\Xi^i\}_{i=0}^n$.
For simplicity, write $\mathcal{G}^\infty$ for the collection of all possible $\Xi$.
%The example in Fig.~\ref{fig:rule-example2} is also used in~\cite{Li2023} where a different property called scale-freeness is studied.
%In that paper, scale-freeness is considered as a special self-similarity in terms of degree distribution.
%These results exhibit the enormous interests and potential properties of random iterated graph systems.
%and particularly let $\Xi$ denotes such graph. 

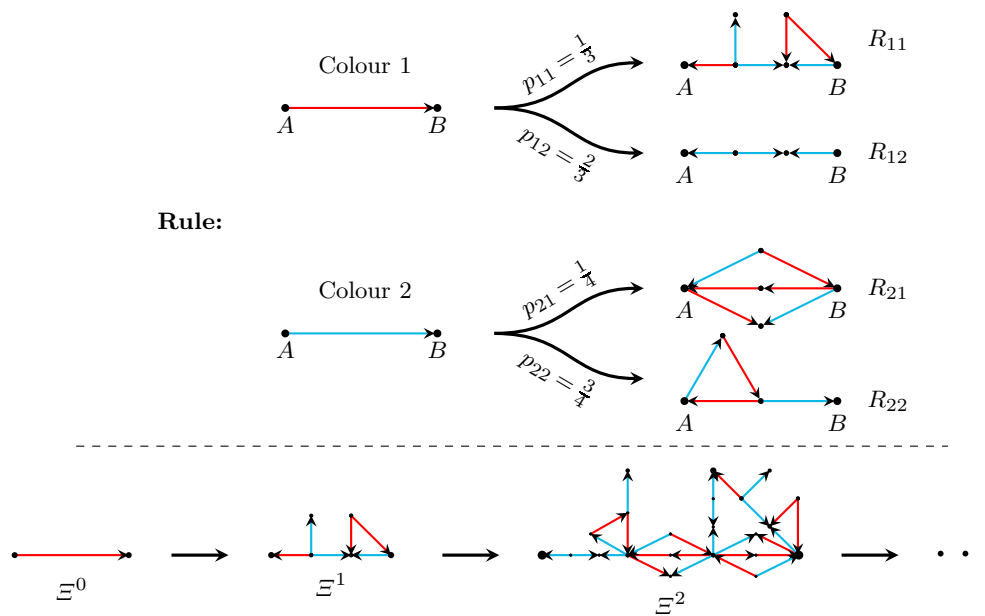
\begin{figure}[ht]
\centering
\begin{tikzpicture}[scale=1,shorten >= 1pt]
  \begin{scope}[scale=1,shift={(0,2)}]
  \draw (-2.5,-1.5) node {\normalsize \textbf{Rule:}};
    \redarc{-1.25,0}{.75,0}
    \draw   (-.2,.8) node[below] {Colour 1};
    \draw  (-1.25,0) node[below] {\footnotesize $A$};
    \draw  (  .75,0) node[below] {\footnotesize $B$};
    %\draw  (-2   ,0) node {$G^0$};
    \draw[very thick,>=stealth,->] (1.5,0) .. controls (2.5,0) and (2.5, .6) .. (3.5, .6) node[sloped,midway,above]{\footnotesize$p_{11} = \frac{1}{3}$};
    \draw[very thick,>=stealth,->] (1.5,0) .. controls (2.5,0) and (2.5,-.6) .. (3.5,-.6) node[sloped,midway,below]{\footnotesize$p_{12} = \frac{2}{3}$};
    \Raa{4, .57}{.67}{0}{\ABnodes\draw (4, .5) node {\makebox(0,0){$R_{11}$}};}
    \Rab{4,-.6 }{.67}{0}{\ABnodes\draw (4,0  ) node {\makebox(0,0){$R_{12}$}};}
  \end{scope}
  \begin{scope}[scale=1,shift={(0,-1)}]
    \bluearc{-1.25,0}{.75,0} 
    \draw   (-.2,.8) node[below] {Colour 2};
    \draw   (-1.25,0) node[below] {\footnotesize $A$};
    \draw   (  .75,0) node[below] {\footnotesize $B$};
    %\draw   (-2   ,0) node {$G^0$};
    \draw[very thick,>=stealth,->] (1.5,0) .. controls (2.5,0) and (2.5, .6) .. (3.5, .6) node[sloped,midway,above]{\footnotesize$p_{21} = \frac{1}{4}$};
    \draw[very thick,>=stealth,->] (1.5,0) .. controls (2.5,0) and (2.5,-.6) .. (3.5,-.6) node[sloped,midway,below]{\footnotesize$p_{22} = \frac{3}{4}$};
    \Rba{4,  .6}{.67}{0}{\ABnodes\draw (4,0) node {\makebox(0,0){$R_{21}$}};}
    \Rbb{4,-.9 }{.67}{0}{\ABnodes\draw (4,0) node {\makebox(0,0){$R_{22}$}};}
    \draw[dashed] (-4,-1.5) -- (7.5,-1.5);
  \end{scope}
\end{tikzpicture}
\begin{tikzpicture}[scale=.75]
  \redarc{0,0}{2,0}
  \draw (1,-.25) node[below] {$\Xi^0$};
  \draw[very thick,>=stealth,->] (2.75,0) --++ (1,0);
  \Raa{4.5,0}{.7}{0}{\draw (1.5,-.25) node[below] {$\Xi^1$};}
  \draw[very thick,>=stealth,->] (7.5,0) --++ (1,0);
  \begin{scope}[scale=1,shift={(7.75,0)}]
    \Rab{1.5,0  }{.5}{0}{}
    \Rba{3  ,0  }{.5}{0}{}
    \Rba{6  ,0  }{.5}{180}{}
    \Rbb{3  ,0  }{.5}{90}{}
    \Rab{4.5,1.5}{.5}{270}{}
    \Raa{4.5,1.5}{.7}{315}{}
    \fill[black] (1.5,0) circle (2.25pt);
    \fill[black] (6  ,0) circle (2.25pt);
    \draw (3.75,-.5) node[below] {$\Xi^2$};
  \end{scope}
  \draw[very thick,>=stealth,->] (14.5,0) --++ (1,0);
  \draw (16.75,0) node(){\huge$\cdots$};
\end{tikzpicture}
\caption{An example of random iterated graph system}
\label{fig:rule-example2}
\end{figure}
%%%%%%%%%%%%%%%%%%%%%%%%%%%%%%%%%%%%%%%%%%%%
\subsection{Relevant definitions}
Please note that the definitions and notations for random iterated graph systems are different from deterministic iterated graph systems.

\begin{definition}
A finite set containing $k \times k$ primitive matrices $\mathcal{X}=\{\mathbf{X}_1,\ldots,\mathbf{X}_m\}$ where every matrix is with a probability vector $(p_1,\ldots,p_m)$ is called a \textbf{random matrices set}.
\end{definition}

\begin{definition}
Let 
\[
\mathcal{M}=\Bigl\{\mathbf{M}=
\begin{pmatrix}
\chi(R_{1 j_1})\\
\cdots \\
\chi(R_{\lambda j_\lambda})
\end{pmatrix}
\::\: j_i = 1,\ldots,q_i \Bigr\}
\]
with $\mathbb{P}(\mathbf{M})=\prod_{i=1}^\lambda p_{ij_i}$.
Clearly $\sum_{\mathbf{M}\in \mathcal{M}} \mathbb{P}(\mathbf{M}) =1$, and hence $\mathcal{M}$ is a random matrices set with $|\mathcal{M}|=\prod_{i=1}^\lambda q_i$ elements.

\end{definition}

\begin{definition}
Let $P_{ij}$ be a path between $A$ and $B$ in $R_{ij}$, and $\mathcal{P}_{ij}$ be the collection of all possible $P_{ij}$ in $\mathcal{R}_{ij}$.
\[
\mathcal{D}=\Bigl\{\mathbf{D}=
\begin{pmatrix}
\chi(P_{1 j_1})\\
\cdots \\
\chi(P_{\lambda j_\lambda})
\end{pmatrix}
\::\:  j_i = 1,\ldots,q_i \Bigr\}
\]
By the definition, $\mathcal{D}$ is a set containing $|\mathcal{D}|=\prod_{i=1}^\lambda q_i$ matrices with
\[
\mathbb{P}\Bigl(
\mathbf{D}=
\begin{pmatrix}
\chi(P_{1 k_1}) \\
\cdots \\
\chi (P_{\lambda k_\lambda})
\end{pmatrix} 
\Big| \mathbf{D} \in \mathcal{D}
\Bigr)
=\prod_{i=1}^\lambda p_{ik_i}\,.
\]
In this way, $\mathcal{D}$ is a random matrices set.

Let
\[
\mathscr{D}=
\Bigl\{
\mathcal{D}\::\: P_{ij_i} \in \mathcal{P}_{ij_i}, \quad i=1,\ldots,\lambda
\Bigr\}\,.
\]
In other words, $\mathscr{D}$ is a collection of all possible random matrices sets $\mathcal{D}$.
Moreover, 
\[
|\mathscr{D}|=\prod_{i=1}^\lambda \prod_{j=1}^{q_i} |\mathcal{P}_{ij}|\,.
\]

If all matrices in $\mathcal{M}$ and $\mathcal{D}\in\mathscr{D}$ are primitive (or positive), we say the random iterated graph $\mathscr{R}$ system is {\bf primitive}.
\end{definition}

\begin{notation}
Let $\mathcal{L}(\mathcal{X})$ be the (maximal) Lyapunov exponent of random matrices set $\mathcal{X}$ with probability vector.
We will strictly define $\mathcal{L}(*)$ in subsection~\ref{subsec:Stochastic Lyapunov}.
For notational convenience, write
\[
\mathcal{L}_{\min} (\mathscr{D}):=
\min_{\mathcal{D}\in \mathscr{D}} \mathcal{L}(\mathcal{D})\,.
\]
\end{notation}

\begin{theorem}[main theorem]\label{thm:Random Substitution}
Given a random primitive iterated graph system and one graph limit $\Xi\in\mathcal{G}^\infty$, we have
\[
\mathbb{P} \Big(
\dim_B (\Xi) =  \dim_H (\Xi)
= \frac{ \mathcal{L}(\mathcal{M})}{ \min_{\mathcal{D}\in\mathscr{D}} \mathcal{L}(\mathcal{D})}
\Big)=1 \,.
\]
\end{theorem}

\begin{remark}
    For any network with a given algebraic number fractal dimension, we can always directly construct a random iterated graph system that possesses the same dimension. 
\end{remark}

\begin{example}
\[
\mathcal{M}=
\biggl\{
\begin{pmatrix}   
3 & 3 \\
4 & 2
\end{pmatrix}
,
\begin{pmatrix}   
3 & 3 \\
2 & 2
\end{pmatrix}
,
\begin{pmatrix}   
0 & 3 \\
4 & 2
\end{pmatrix}
,
\begin{pmatrix}   
0 & 3 \\
2 & 2
\end{pmatrix}
\biggr\}\,,
\]
respectively with the probability vector $(\frac{1}{12},\frac{1}{4},\frac{1}{6},\frac{1}{2})$.
The Lyapunov exponent 
\[
\mathcal{L}(\mathcal{M})\approx1.4488\,.
\]

Hence, in this example, $\mathscr{D}$ contains $|\mathcal{P}_{11}|\times|\mathcal{P}_{12}|\times|\mathcal{P}_{21}|\times|\mathcal{P}_{22}|=2\times1\times3\times2=12$ random matrices sets.
\begin{align*}
&\qquad\mathscr{D}= \\
&\Biggl\{
\biggl\{
\begin{pmatrix}   
1 & 2 \\
1 & 1
\end{pmatrix}
,
\begin{pmatrix}   
1 & 2 \\
1 & 1
\end{pmatrix}
,
\begin{pmatrix}   
0 & 3 \\
1 & 1
\end{pmatrix}
,
\begin{pmatrix}   
0 & 3 \\
1 & 1
\end{pmatrix}
\biggr\}\,,
%%%%%%%%%%%%%%%%%
\biggl\{
\begin{pmatrix}   
3 & 1 \\
1 & 1
\end{pmatrix}
,
\begin{pmatrix}   
3 & 1 \\
1 & 1
\end{pmatrix}
,
\begin{pmatrix}   
0 & 3 \\
1 & 1
\end{pmatrix}
,
\begin{pmatrix}   
0 & 3 \\
1 & 1
\end{pmatrix}
\biggr\}\,,\\
%%%%%%%%%%%%%%%%%%%%%
&\biggl\{
\begin{pmatrix}   
3 & 1 \\
2 & 0
\end{pmatrix}
,
\begin{pmatrix}   
3 & 1 \\
1 & 1
\end{pmatrix}
,
\begin{pmatrix}   
0 & 3 \\
2 & 0
\end{pmatrix}
,
\begin{pmatrix}   
0 & 3 \\
1 & 1
\end{pmatrix}
\biggr\}\,,
%%%%%%%%%%%%%%%%%%%%%%%
\biggl\{
\begin{pmatrix}   
1 & 2 \\
2 & 0
\end{pmatrix}
,
\begin{pmatrix}   
1 & 2 \\
1 & 1
\end{pmatrix}
,
\begin{pmatrix}   
0 & 3 \\
2 & 0
\end{pmatrix}
,
\begin{pmatrix}   
0 & 3 \\
1 & 1
\end{pmatrix}
\biggr\}\,,
\cdots
\Biggr\}
\end{align*}
all with the probability vector $(\frac{1}{12},\frac{1}{4},\frac{1}{6},\frac{1}{2})$.
The Lyapunov exponents of elements show in $\mathscr{D}$ above are $\{0.8717,0.9474,0.9705,0.8900,\cdots\}$ respectively.
Finding the smallest Lyapunov exponent in $\mathscr{D}$, it is \[
\mathcal{L}_{\min}(\mathscr{D})\approx0.8717 \,.
\]

By Theorem~\ref{thm:Random Substitution}, we have the theoretical dimension
\[
\dim_B(\Xi)\overset{a.e.}{=}\frac{\mathcal{L}(\mathcal{M})}{\mathcal{L}_{\min}(\mathscr{D})}\approx\frac{1.4488}{0.8717}=1.6620\,.
\]
Also, modelling the limit by the volume-greedy ball-covering algorithm (VGBC) as given by Wang et al.~\cite{WaWaXiChWaBaYuZh17},
we obtain the simulated values given in Fig.~\ref{fig:10 simulations}.
It shows 10 times of simulating Minkowski dimensions of $\Xi^5$, providing a estimated Minkowski dimension~$1.4275$.

\begin{figure}
\centering
\pgfplotstableread{
X1	Y1	X2	Y2	X3	Y3	X4	Y4	X5	Y5	X6	Y6	X7	Y7	X8	Y8	X9	Y9	X10	Y10
2	0.1483	2	0.1446	2	0.1464	2	0.1423	2	0.1513	2	0.1478	2	0.1407	2	0.1469	2	0.1465	2	0.1487
3	0.0942	3	0.093	3	0.084	3	0.0923	3	0.0811	3	0.0853	3	0.0906	3	0.0952	3	0.094	3	0.0855
4	0.0697	4	0.0659	4	0.0641	4	0.0663	4	0.0666	4	0.0613	4	0.0665	4	0.0707	4	0.0678	4	0.0632
5	0.0463	5	0.0434	5	0.0441	5	0.0423	5	0.0472	5	0.0445	5	0.0462	5	0.0449	5	0.0459	5	0.0458
6	0.0368	6	0.0387	6	0.0316	6	0.0365	6	0.04	6	0.0373	6	0.0356	6	0.0395	6	0.0364	6	0.0347
7	0.0329	7	0.0326	7	0.0316	7	0.0308	7	0.0375	7	0.0337	7	0.0337	7	0.0354	7	0.0321	7	0.0335
8	0.0256	8	0.0231	8	0.02	8	0.024	8	0.023	8	0.0228	8	0.027	8	0.0218	8	0.0204	8	0.0248
9	0.0217	9	0.019	9	0.0183	9	0.0192	9	0.0194	9	0.018	9	0.025	9	0.0204	9	0.0175	9	0.0235
10	0.0173	10	0.0156	10	0.0133	10	0.0154	10	0.0169	10	0.0156	10	0.0202	10	0.0177	10	0.0153	10	0.0173
11	0.0145	11	0.0129	11	0.0116	11	0.0135	11	0.0157	11	0.0144	11	0.0154	11	0.0163	11	0.0124	11	0.0161
12	0.0128	12	0.0122	12	0.0108	12	0.0125	12	0.0133	12	0.0132	12	0.0154	12	0.0122	12	0.0102	12	0.0136
13	0.0106	13	0.0095	13	0.0108	13	0.0125	13	0.0133	13	0.012	13	0.0154	13	0.0136	13	0.0102	13	0.0124
14	0.01	14	0.0095	14	0.0092	14	0.0115	14	0.0085	14	0.0096	14	0.0116	14	0.0109	14	0.0095	14	0.0099
15	0.0084	15	0.0088	15	0.0092	15	0.0106	15	0.0085	15	0.0096	15	0.0106	15	0.0109	15	0.0087	15	0.0112
16	0.0089	16	0.0088	16	0.0083	16	0.0077	16	0.0085	16	0.0084	16	0.0096	16	0.0095	16	0.0087	16	0.0099
17	0.0072	17	0.0088	17	0.0083	17	0.0077	17	0.0085	17	0.0084	17	0.0087	17	0.0082	17	0.0087	17	0.0099
18	0.0072	18	0.0095	18	0.0083	18	0.0067	18	0.0085	18	0.0084	18	0.0077	18	0.0082	18	0.008	18	0.0087
19	0.0072	19	0.0095	19	0.0058	19	0.0038	19	0.0073	19	0.0084	19	0.0058	19	0.0082	19	0.0066	19	0.0074
20	0.0072	20	0.0075	20	0.0058	20	0.0038	20	0.0048	20	0.006	20	0.0058	20	0.0082	20	0.0051	20	0.0062
21	0.0056	21	0.0068	21	0.005	21	0.0038	21	0.0048	21	0.006	21	0.0058	21	0.0068	21	0.0044	21	0.005
22	0.0056	22	0.0061	22	0.005	22	0.0038	22	0.0048	22	0.006	22	0.0048	22	0.0068	22	0.0044	22	0.005
23	0.0056	23	0.0048	23	0.005	23	0.0038	23	0.0048	23	0.006	23	0.0039	23	0.0068	23	0.0044	23	0.005
24	0.0045	24	0.0048	24	0.005	24	0.0038	24	0.0048	24	0.0048	24	0.0039	24	0.0054	24	0.0044	24	0.005
25	0.0045	25	0.0048	25	0.005	25	0.0038	25	0.0048	25	0.0048	25	0.0039	25	0.0054	25	0.0044	25	0.005
26	0.0039	26	0.0054	26	0.005	26	0.0038	26	0.0036	26	0.0048	26	0.0039	26	0.0041	26	0.0044	26	0.005
27	0.0039	27	0.0048	27	0.005	27	0.0029	27	0.0036	27	0.0036	27	0.0039	27	0.0041	27	0.0044	27	0.005
28	0.0039	28	0.0048	28	0.005	28	0.0029	28	0.0036	28	0.0036	28	0.0039	28	0.0041	28	0.0036	28	0.005
29	0.0039	29	0.0041	29	0.005	29	0.0019	29	0.0024	29	0.0036	29	0.0039	29	0.0041	29	0.0036	29	0.005
30	0.0033	30	0.0034	30	0.0042	30	0.0019	30	0.0024	30	0.0024	30	0.0039	30	0.0041	30	0.0036	30	0.005

}\mytable
\begin{tikzpicture}[scale=0.8]
    \begin{axis}[
        xmode=log,
        ymode=log,
        xmin = 1.5, xmax = 40,
        ymin = 0.001, ymax = 1,
        width = 0.9\textwidth,
        height = 0.675\textwidth,
       % xtick distance = 1,
       % ytick distance = 1,
        grid = both,
        minor tick num = 1,
        major grid style = {lightgray},
        minor grid style = {lightgray!25},
        xlabel = {$L$},
        ylabel = {$\frac{N_L(\Xi^n)}{|V(\Xi^n)|}$},
        legend cell align = {left},
        legend pos = north east
        ]
        %\foreach \x in;
        \addlegendentry{Estimated Minkowski dimension is $1.4275$};
        \addplot[color=blue,mark=sqaure,only marks] 
        table[x = X1, y = Y1] {\mytable};
        \addplot[thick, black] table[
            x = X1,
            y = {create col/linear regression={y=Y1}}
            ] {\mytable};
        \addplot[color=red,mark=square,only marks] table[x = X2, y = Y2] {\mytable};
        \addplot[thick, black] table[ x = X2, y = {create col/linear regression={y=Y2}} ] {\mytable};
        \addplot[color=green,mark=square,only marks] table[x = X3, y = Y3] {\mytable};
        \addplot[thick, black] table[ x = X3, y = {create col/linear regression={y=Y3}} ] {\mytable};
        \addplot[color=cyan,mark=square,only marks] table[x = X4, y = Y4] {\mytable};
        \addplot[thick, black] table[ x = X4, y = {create col/linear regression={y=Y4}} ] {\mytable};
        \addplot[color=magenta,mark=square,only marks] table[x = X5, y = Y5] {\mytable};
        \addplot[thick, black] table[ x = X5, y = {create col/linear regression={y=Y5}} ] {\mytable};
        \addplot[color=yellow,mark=square,only marks] table[x = X6, y = Y6] {\mytable};
        \addplot[thick, black] table[ x = X6, y = {create col/linear regression={y=Y6}} ] {\mytable};
        \addplot[color=violet,mark=square,only marks] table[x = X7, y = Y7] {\mytable};
        \addplot[thick, black] table[ x = X7, y = {create col/linear regression={y=Y7}} ] {\mytable};
        \addplot[color=gray,mark=square,only marks] table[x = X8, y = Y8] {\mytable};
        \addplot[thick, black] table[ x = X8, y = {create col/linear regression={y=Y8}} ] {\mytable};
        \addplot[color=purple,mark=square,only marks] table[x = X9, y = Y9] {\mytable};
        \addplot[thick, black] table[ x = X9, y = {create col/linear regression={y=Y9}} ] {\mytable};
        \addplot[color=brown,mark=square,only marks] table[x = X10, y = Y10] {\mytable};
        \addplot[thick, black] table[ x = X10, y = {create col/linear regression={y=Y10}} ] {\mytable};
    \end{axis}
\end{tikzpicture}
\caption{$10$ simulations of fractality when $n=5$}
\label{fig:10 simulations}
\end{figure}
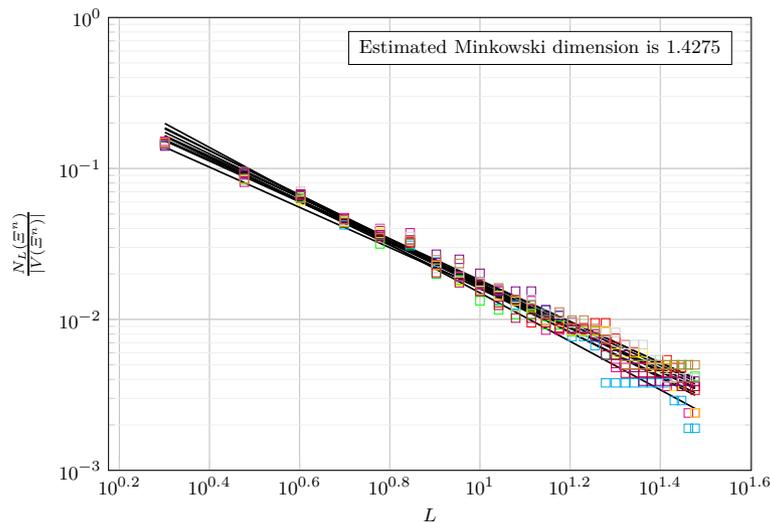

\end{example}

%%%%%%%%%%%%%%%%%%%%%%%%%%%%%%%%%%%%%%%%%%%%

%%%%%%%%%%%%%%%%%%%%%%%%%%%%%%%%%%%%%%%%%%%%

\subsection{Lyapunov exponents and stochastic substitution systems}
\label{subsec:Stochastic Lyapunov}
This subsection is devoted to studying the stochastic behaviours of random iterated graph systems.
We summarise the process of random substitution on graph fractals to so-called \textit{stochastic substitution systems}, and we prove some properties of the associated Lypunov exponents.
Therefore, we shall start from introducing Lyapunov exponents.

\bigskip

Lyapunov exponents regarding random matrices product were first studied by Bellman~\cite{Bell54}.
Furstenberg and Kesten~\cite{Fursten60,Fursten63} then developed couples of famous results of the asymptotic behaviours for random matrices products.
Recently, Pollicott~\cite{Pollicott10,Pollicott21} presents several impressive works on effective numerical estimates of Lyapunov exponents.

We will continue to use the notations in~\cite{Pollicott10} for Lyapunov exponents.
Let $\mathcal{X}=\{\mathbf{X}_1,\ldots,\mathbf{X}_m\}$ be a random primitive and invertible matrices set with probability vector $(p_1,\dots,p_m)$.
Consider the sequence space $\Sigma=\{1,\ldots,m\}^{\mathbb{Z}^+}$, and $\mu$ be the corresponding Bernoulli measure on such space.
Choosing $\underline{i}=(i_n)_{n=1}^\infty \in \Sigma$,
the Lyapunov exponent of random matrices products is defined as
\begin{align}\label{formula:Lyapunov1}
\mathcal{L}(\mathcal{X})
&:=\lim_{n\to \infty}\frac{1}{n}\mathbb{E}\bigl(\log \Vert \mathbf{X}_{i_1} \cdots \mathbf{X}_{i_n}\Vert  \bigr) 
=\lim_{n\to \infty} \frac{1}{n} \int_{\underline{i}\in \Sigma} \log
\Vert \mathbf{X}_{i_1} \cdots \mathbf{X}_{i_n}\Vert d\mu(\underline{i}) \,.
\end{align}

Furstenberg and Kesten~\cite{Fursten60} develop that if $X_{i_n}$ is a stationary stochastic process, for almost every $(\mu)$ $\underline{i} \in \Sigma$, there is
\begin{align}\label{formula:Lyapunov2}
\mathcal{L}(\mathcal{X}) 
=\lim_{n \to \infty} \frac{1}{n} \log \Vert \mathbf{X}_{i_1} \cdots \mathbf{X}_{i_n} \Vert
=\lim_{n \to \infty} \frac{1}{n} \log [ \mathbf{X}_{i_1} \cdots \mathbf{X}_{i_n} ]_{jk}
\,,
\end{align}
for any $jk$ entry.
In this subsection, $\Vert*\Vert$ is valid for any norm.

In this paper, we will discuss a type of slightly different matrices products, but it will be proved the Lyapunov exponent of such system coincides with the associated random matrices products.

\bigskip

%\begin{proof}
%As $\mathbf{X}$ is primitive, Perron-Frobenius Theorem~\cite{Mey00} provides a spectral projector, a finite matrix $\mathbf{G}$. Given that $\rho(\mathbf{G})=\rho(\mathbf{X})$, the result is thus proved.
%\end{proof}

\begin{definition}\label{def:Stochastic substitution System}{\rm
Let $\mathcal{X}$ be a random matrices set with probability vector $Q=(p_1,\ldots,p_m)$.
For each $i=1,\ldots,k$, let $\mathbf{e}_{i}$ be the $i$-th standard basis unit vector of~$\mathbb{R}^k$.

Define a random function $\psi \::\: \{ \mathbf{e}_{1}, \dots, \mathbf{e}_{k}\} \to \mathbb{Z}_+^k$.
Define $k$ identical and independent $\psi(\mathbf{e}_1),\ldots,\psi(\mathbf{e}_k)$,
each with probability function
\[
  \mathbb{P}(\psi(\mathbf{e}_{i}) = \mathbf{e}_{i} \mathbf{X}_j) = p_j \,.
\]
Here, the random function $\psi$ randomly maps a basis unit vector to some vector by a certain probability function.
Consider any non-negative vector $\mathbf{x} = [\mathbf{x}]_1\mathbf{e}_1+\cdots+[\mathbf{x}]_k\mathbf{e}_k \in\mathbb{Z}_+^k$,
and define the random function $\Psi \::\: \mathbb{Z}_+^k  \to \mathbb{Z}_+^k$ by
\[
  \Psi(\mathbf{x}) = \overbrace{\psi(\mathbf{e}_1)+ \cdots+\psi(\mathbf{e}_1)}^{x_1}  + \cdots + \overbrace{\psi(\mathbf{e}_k)+\cdots +\psi(\mathbf{e}_n)}^{x_k} \,.
\]
Let $\boldsymbol{\alpha}_0$ be a given non-negative vector with integer entries
and define a sequence of stochastic vectors $\{\boldsymbol{\alpha}_n\}_{n\geq 1}$
by the iterations
\[
%  \boldsymbol{\alpha}_t =   \boldsymbol{\alpha}_{t-1}\overline{X} \qquad\text{and}\qquad
  \boldsymbol{\alpha}_n  = \Psi(\boldsymbol{\alpha}_{n-1}) \,.
\]
%Here clearly $\mathbb{E}(\boldsymbol{\beta}_t)=\boldsymbol{\beta}_{t-1}\overline{X}$ if $\boldsymbol{\beta}_{t-1}$ is fixed.
%Finally let $\xi_t = \Vert  \boldsymbol{\alpha}_t \Vert_1$ for all $t \in \mathbb{N}$. 
%%In this thesis there is always $t \in \mathbb{N}$.
%and note that $\mathbb{E}(\xi_t)= | \mathbb{E}(\boldsymbol{\beta}_t) |_1= | \boldsymbol{\alpha}_t |_1$.
We call such the triple $\mathcal{S}=(\Vert\boldsymbol{\alpha}_{n}\Vert,\boldsymbol{\alpha}_{0},\mathcal{X})$ a {\em stochastic substitution system}.
Let $(\Omega,\mathcal{F}_n,\mathbb{P})$ be a probability space with $\Omega$ being a sample space.}
\end{definition}

\begin{definition}
Let $(\Vert \boldsymbol{\alpha}_{t}\Vert,\boldsymbol{\alpha}_{0},\mathcal{X})$ be a stochastic substitution system.
The Lyapunov exponent of $\mathcal{S}$ is defined by
\[
\mathcal{L}(\mathcal{S})
:=\lim_{n \to \infty} \frac{1}{n} \mathbb{E} \bigl( \log \Vert \boldsymbol{\alpha}_n \Vert \bigr) \,,
\]
where $\mathbb{E}$ is the expectation. 
\end{definition}

\begin{theorem}\label{thm:Lyapunov}
Let $\mathcal{S}=(\Vert \boldsymbol{\alpha}_{n}\Vert,\boldsymbol{\alpha}_{0},\mathcal{X})$ be a stochastic substitution system.
\[
\mathcal{L(\mathcal{S})}=\mathcal{L}(\mathcal{X})\,.
\]
\end{theorem}
\begin{proof}
First, $\boldsymbol{\alpha}_0$ is a non-negative row vector.
By formula (\ref{formula:Lyapunov2}), we have
\[
\lim_{n \to \infty} \frac{1}{n} \log \Vert \mathbf{X}_{i_1} \cdots \mathbf{X}_{i_n} \Vert
=\lim_{n \to \infty} \frac{1}{n} \log \Vert \boldsymbol{\alpha}_0 \mathbf{X}_{i_1} \cdots \mathbf{X}_{i_n} \Vert\,.
\]

Because $X_{i_n}$ are all independent, it is found 
\begin{align*}
\mathbb{E}\bigl( \log \Vert \boldsymbol{\alpha}_0 \mathbf{X}_{i_1} \cdots \mathbf{X}_{i_n}\Vert \bigr)
&=\log \Vert\mathbb{E}( \boldsymbol{\alpha}_0 \mathbf{X}_{i_1} \cdots \mathbf{X}_{i_n})\Vert \\
&=\log \Vert\mathbb{E}( \Phi({\boldsymbol{\alpha}}_0))\mathbb{E}( \mathbf{X}_{i_2} \cdots \mathbf{X}_{i_n})\Vert
=\log \Vert\mathbb{E}( \boldsymbol{\alpha}_1)\mathbb{E}( \mathbf{X}_{i_2} \cdots \mathbf{X}_{i_n})\Vert \\
&= \cdots =\log \Vert \mathbb{E}(\boldsymbol{\alpha}_n)\Vert \,.
\end{align*}

Finally, by the above,
\begin{align*}
\mathcal{L}(\mathcal{X})
&=\lim_{n\to \infty}\frac{1}{n}\mathbb{E}\bigl(\log \Vert \mathbf{X}_{i_1} \cdots \mathbf{X}_{i_n}\Vert  \bigr) \\
&=\lim_{n \to \infty} \frac{1}{n} \log \Vert \mathbb{E}(\boldsymbol{\alpha}_n)\Vert
=\lim_{n \to \infty} \frac{1}{n} \mathbb{E}(\log \Vert \boldsymbol{\alpha}_n\Vert) \\
&=\mathcal{L}(\mathcal{S})\,.
\end{align*}
\end{proof}

\begin{theorem}\label{thm:Stochastic Substitution System}
Almost surely, 
\[
\Vert \boldsymbol{\alpha}_n \Vert \asymp \exp\bigl(\mathcal{L}(\mathcal{X})\bigr)^n\,.
\]
\end{theorem}
\begin{proof}
By formula~(\ref{formula:Lyapunov2}) and Theorem~\ref{thm:Lyapunov}.
\end{proof}

\smallskip

\subsection{Growths of arcs and nodes}

\begin{lemma}\label{lem:Random Arc Growth}
Almost surely $|E(\Xi^n)|\overset{n \to \infty}{\asymp} 
\exp\bigl(\mathcal{L}(\mathcal{M})\bigr)^n$.
\end{lemma}
\begin{proof}
$(|E(\Xi^n)|=\Vert \chi(\Xi^n) \Vert_1,\chi(\Xi^0),\mathcal{M})$ is clearly a stochastic substitution system because every arc substitution is independent and associated with some matrix in $\mathcal{M}$.
Then by Theorem~\ref{thm:Stochastic Substitution System} we obtain the desired result.
\end{proof}

\begin{lemma}\label{lem:Random Node Growth}
Almost surely $|V(\Xi^n)|\overset{n \to \infty}{\asymp} 
\exp\bigl(\mathcal{L}(\mathcal{M})\bigr)^n$.
\end{lemma}
\begin{proof}
This proof is of the form of Lemma~\ref{lem:Node Growth}, by Lemma~\ref{lem:Random Arc Growth} providing that almost surely $|E(\Xi)|\overset{n \to \infty}{\asymp} 
\exp\bigl(\mathcal{L}(\mathcal{M})\bigr)^n$.
\end{proof}

\begin{theorem}\label{thm:sparse}
$\Xi$ is always a sparse graph.
\end{theorem}
\begin{proof}
For all possible $\Xi\in \mathcal{G}^\infty$, there are always
\[
D(\Xi):=\lim_{n\to\infty}\frac{|E(\Xi^n)|}{|V(\Xi^n)|(|V(\Xi^n)|-1)}=0\,.\qquad%\qedhere
\]
\end{proof}
\begin{remark}
Theorem~\ref{thm:sparse} holds even if all graphs in $\mathcal{R}$ are complete.
\end{remark}

%%%%%%%%%%%%%%%%%%%%%%%%%%%%%%%%%%%%%%%%%%%%%%%%%
\subsection{Proof of Theorem~\ref{thm:Random Substitution}}\label{sec:Proof of Random}

\begin{lemma}\label{lem:RandomGamma}
$\mathbb{P}\Bigl(|\Gamma^n| \asymp \exp\bigl(\mathcal{L}_{\min}(\mathscr{D})\bigr)^n\Bigr) =1 $.
\end{lemma}
\begin{proof}
For a certain $\mathcal{D}\in \mathscr{D}$, we verify that $(\Vert\chi(\Gamma^n)\Vert_1,\chi(\Gamma^0),\mathcal{D})$ is a stochastic substitution system.
This is because every ``kind" of independent substitution $\mathcal{D}$ can be found within some $\Gamma^n$.
By Lemma~\ref{lem:Uniform}, there exists one $\mathcal{D}$ such that $\Gamma^n$ is almost surely shortest.
In other words, for every rule graph $R_{ij}$, even though there are (one or) many possible paths between $A$ and $B$,
Lemma~\ref{lem:Uniform} always guarantees that an optimisation can be achieved if we only pick one particular path between $A$ and $B$ from each $R_{ij}$. 

Hence, the paths between $A$ and $B$ in $\Xi^n$ all possess growth rate $\exp\bigl(\mathcal{L}(\mathcal{D})\bigr)$.
Among them, the shortest path $\Gamma^n$ should follow the growth rate $\exp\bigl(\mathcal{L}_{\min}(\mathscr{D})\bigr)$.

\end{proof}

\begin{lemma}\label{lem:RandomDiameter}
$\mathbb{P} \Bigl(\Delta(\Xi^n) \asymp \exp\bigl(\mathcal{L}_{\min}(\mathscr{D})\bigr)^n \Bigr)=1 $.
\end{lemma}
\begin{proof}
The analogous proof is established like Lemma~\ref{lem:Diameter}, by Lemma~\ref{lem:RandomGamma}.

\end{proof}

\bigskip

\textbf{We are now to prove Theorem~\ref{thm:Random Substitution}.}
\begin{proof}
We start from verifying the Minkowski dimension.
By Lemma~\ref{lem:RandomDiameter}, the growth of diameter of $\Xi^n$ is almost surely $\exp\bigl(\mathcal{L}_{\min}(\mathscr{D})\bigr)$.
It is feasible that we use the tricks in the proof of Theorem~\ref{thm:part I} in Subsection~\ref{Subsection: Proof of fractality}.
Similarly, we obtain that $|E(\Xi^k)|\preceq N_L^* (\Xi^n) \leq N_L (\Xi^n) \preceq |E(\Xi^k)|$ holds for almost every $\Xi\in \mathcal{G}$ when $n \to \infty$.
This suffices to have $\dim_B (\Xi)=\frac{\mathcal{L}(\mathcal{M})}{\mathcal{L}_{\min}(\mathscr{D})}$ with probability $1$.

In the sense of Hausdorff dimension, we also adopt the definition in Subsection~\ref{Subsection:Proof of Hausdorff}.
With the same techniques to estimate the mass distribution, the result boils down to, almost surely, $\dim_H(\Xi)=\frac{\mathcal{L}(\mathcal{M})}{\mathcal{L}_{\min}(\mathscr{D})}$.

Consequently, for almost every $\Xi \in \mathcal{G}^\infty$, one has $\dim_B(\Xi)=\dim_H(\Xi)=\frac{\mathcal{L}(\mathcal{M})}{\mathcal{L}_{\min}(\mathscr{D})}$, which completes the final proof.
\end{proof}

%%%%%%%%%%%%%%%%%%%%%%%%%%%%%%%%%%%%%%%%%%%%%
\section*{Acknowledgements}
This work was supported by the Additional Funding Programme for Mathematical Sciences, delivered by EPSRC (EP/V521917/1) and the Heilbronn Institute for Mathematical Research,
and also by the EPSRC Centre for Doctoral Training in Mathematics of Random Systems: Analysis, Modelling and Simulation (EP/S023925/1).
%The author is a recipient of Heilbronn HDP Scholarship funded by Engineering and Physical Sciences Research Council (EPSRC), Heilbronn Institute for Mathematical Research (HIMR) and Mathematics Department of Imperial College London.

%%%%%%%%%%%%%%%%%%%%%%%%%%%%%%%%%%%%%%%%%%
%\bibliographystyle{ieeetr}
% \bibliography{Reference}

\bibliography{Reference}
\bibliographystyle{RS}

% \affiliationone{
%    Nero Ziyu Li\\
%    Department of Mathematics, Imperial College London, South Kensington Campus, London SW7 2AZ, United Kingdom.
%    \email{z5222549@zmail.unsw.edu.au\\
%    ziyu.li21@imperial.ac.uk}}
% Important: Do not put any empty line here.

\end{document}